\documentclass[a4paper]{article}

\usepackage[utf8]{inputenc} 
\usepackage[T1]{fontenc}
\usepackage[english]{babel}
\usepackage{a4}
\usepackage{amsfonts, amssymb, amsmath, amsthm}
\usepackage{mathtools} 
\usepackage{enumitem}
\usepackage{tensor} 
\usepackage[unicode=true]{hyperref}
\hypersetup{pdfauthor={Seerp Roald Koudenburg}, pdftitle={Left Kan extensions that are algebraic over colax-idempotent 2-monads}}

\usepackage{tikz} 
\usetikzlibrary{decorations.markings, matrix, arrows, external, fit, calc}

\tikzstyle{map} = [->, font=\scriptsize]
\tikzstyle{linj} = [left hook->, font=\scriptsize]
\tikzstyle{rinj} = [right hook->, font=\scriptsize]
\tikzstyle{sur} = [->>, font=\scriptsize]
\tikzstyle{cell} = [double,double equal sign distance,-implies, shorten >= 3.75pt, shorten <= 3.75pt, font=\scriptsize]
\tikzstyle{eq} = [double,double equal sign distance]
\tikzstyle{ps} = [shorten >= 2pt]

\tikzstyle{iso} = [above, sloped, inner sep=1.5pt]
\tikzstyle{nat} = [above, sloped, inner sep=2pt]
\tikzstyle{desc} = [fill=white, inner sep=2pt]

\tikzstyle{small} = [font=\scriptsize]

\tikzstyle{textbaseline} = [baseline=-2.8pt]

\tikzstyle{barred} = [decoration={markings, mark=at position 0.5 with {\draw[-] (0,-1.5pt) -- (0,1.5pt);}}, postaction ={decorate}]

\tikzstyle{math4} = [matrix of math nodes, row sep={3.5em,between origins}, column sep={4em,between origins}, text height=1.5ex, text depth=0.25ex, nodes in empty cells]
\tikzstyle{math35} = [math4, row sep={3.25em,between origins}, column sep={3.5em,between origins}]
\tikzstyle{minimath} = [matrix of math nodes, row sep={3em,between origins}, column sep={3.25em,between origins}, font=\scriptsize, text height=1ex, text depth=0.25ex, nodes in empty cells]

\makeatletter
\def\slashedarrowfill@#1#2#3#4#5{%
  $\m@th\thickmuskip0mu\medmuskip\thickmuskip\thinmuskip\thickmuskip
   \relax#5#1\mkern-7mu%
   \cleaders\hbox{$#5\mkern-2mu#2\mkern-2mu$}\hfill
   \mathclap{#3}\mathclap{#2}%
   \cleaders\hbox{$#5\mkern-2mu#2\mkern-2mu$}\hfill
   \mkern-7mu#4$%
}
\def\rightslashedarrowfill@{%
  \slashedarrowfill@\relbar\relbar\mapstochar\rightarrow}
\newcommand\xslashedrightarrow[2][]{%
  \ext@arrow 0055{\rightslashedarrowfill@}{#1}{#2}}
\makeatother

\def\slashedrightarrow{\xslashedrightarrow{}}

\providecommand{\opcart}{\strut opcart}
\providecommand{\cart}{\strut cart}

\newtheorem{theorem}{Theorem}
\newtheorem{lemma}[theorem]{Lemma}
\newtheorem{corollary}[theorem]{Corollary}
\newtheorem{proposition}[theorem]{Proposition}

\theoremstyle{definition}
\newtheorem{definition}[theorem]{Definition}

\theoremstyle{remark}

\newtheorem{example}[theorem]{Example}

\providecommand{\defref}[1]{Definition~\ref{#1}}
\providecommand{\exref}[1]{Example~\ref{#1}}
\providecommand{\figref}[1]{Figure~\ref{#1}}
\providecommand{\lemref}[1]{Lemma~\ref{#1}}
\providecommand{\propref}[1]{Proposition~\ref{#1}}

\providecommand{\thmref}[1]{Theorem~\ref{#1}}

\providecommand{\of}{\circ}
\providecommand{\iso}{\cong}

\providecommand{\brar}{\slashedrightarrow}
\providecommand{\xrar}[1]{\xrightarrow{#1}}

\providecommand{\xRar}[1]{\xRightarrow{#1}}

\providecommand{\eps}{\varepsilon}

\DeclareMathOperator{\dash}{--}
\providecommand{\ndash}{\nobreakdash-}

\providecommand{\tens}{\otimes}

\providecommand{\ul}[1]{\underline{#1}{}}

\providecommand{\mf}[1]{\mathfrak{#1}}

\providecommand{\brcs}[1]{\lbrace #1 \rbrace}

\providecommand{\brks}[1]{\lbrack #1 \rbrack}

\providecommand{\bigpars}[1]{\bigl(#1\bigr)}

\providecommand{\set}[1]{\brcs{#1}}

\providecommand{\natarrow}{\Rightarrow}

\providecommand{\map}[3]{#1\colon#2\to#3}

\providecommand{\nat}[3]{#1\colon#2\natarrow#3}
\providecommand{\cell}[3]{#1\colon#2\Rightarrow#3}

\providecommand{\hmap}[3]{#1\colon#2\slashedrightarrow#3}

\DeclareMathOperator{\id}{id}

\DeclareMathOperator{\yon}{y}

\providecommand{\ladj}{\dashv}

\providecommand{\op}[1]{#1^\textup{op}}
\providecommand{\co}[1]{#1^\textup{co}}

\providecommand{\ps}[1]{\widehat{#1}}
\providecommand{\el}[1]{\int #1}

\providecommand{\catvar}[1]{\mathcal{#1}}

\providecommand{\E}{\catvar E}
\renewcommand{\L}{\catvar L}
\providecommand{\K}{\catvar K}

\providecommand{\V}{\catvar V}

\providecommand{\Set}{\mathsf{Set}}
\providecommand{\Cat}{\mathsf{Cat}}

\providecommand{\Span}[1]{\mathsf{Span}(#1)}

\providecommand{\Prof}{\mathsf{Prof}}
\providecommand{\enProf}[1]{#1\text-\Prof}
\providecommand{\inProf}[1]{\Prof(#1)}

\providecommand{\hc}{\odot}

\begin{document}
	\title{Left Kan extensions that are algebraic over colax-idempotent $2$-monads}
	\author{Seerp Roald Koudenburg}
	\date{Draft version as of December 11, 2014\footnote{I intend to, after further expansion, publish the material in this note as a paper.}}
	\maketitle
	\begin{abstract}
		Using the language of double categories we generalise a classical result on finite-product-preserving left Kan extensions, by Ad\'amek and Rosick\'y, to one on left Kan extensions that preserve algebraic structures defined by `suitable' colax-idempotent $2$-monads, as well as obtain two related results. To be precise, by `suitable' $2$-monads here we mean ones that extend to normal lax double monads. In an appendix we consider induced algebra structures on presheaf objects.
	\end{abstract}
	
	Having worked with double categories for some time now I keep being surprised by their usefulness in studying Kan extensions. Recently I have been looking at how to reformulate a classic result, on finite-product-preserving left Kan extensions, in the language of double categories. This again turned out---in my opinion---very nicely, prompting me to write this note.
	
	The classic result in question is (a part of) Theorem 2.6 of \cite{Adamek-Rosicky01}, by Ad\'amek and Rosick\'y, that can be stated using the notion of cosifted category, as follows. One way of defining a small category $A$ to be \emph{cosifted} is to require that it is not empty and that, for each pair $x$, $y \in A$, the category $\Span{x,y}$, of spans $x \leftarrow \cdot \to y$ in $A$ and their morphisms, is connected.	Writing $\map{\yon_A}A{\ps A = \brks{\op A, \Set}}$ for the yoneda embedding and $\el d$ for the category of elements of a copresheaf $\map dA\Set$, the result is as follows.
	
	\begin{theorem}[Ad\'amek and Rosick\'y] \label{finite-product-preserving left Kan extensions}
		The left Kan extension $\map{\textup{lan}_{\yon_A} d}{\ps A}\Set$ of $\map dA\Set$ along $\map{\yon_A}A{\ps A}$ preserves finite products precisely if $\el d$ is cosifted.
	\end{theorem}
	
	The aim of this note is to reformulate this result in the language of double categories and, by doing so, discover that it generalises to other left Kan extensions, between objects that are similarly algebraic (more precisely: that are algebraic over `suitable' colax-idempotent $2$-monads).
	
	I have tried to keep this note largely self-contained. In its first three sections we review the notions of `colax-idempotent $2$-monads', `double categories' and the `(op-)cartesian cells' therein. Readers familiar with these notions may want to skip these sections. The three sections that follow introduce some category theory within a double category: we will consider `Kan extensions', the `Beck-Chevalley condition' and `yoneda embeddings' in a double category. In the final section we state and prove the generalisation of the theorem above together with two closely related results.
		
	\section{Colax-idempotent 2-monads}
	The first thing we need is that `a choice of finite products $(x_1 \times \dotsb \times x_n)$' in a category $M$, one for each sequence $x_1, \dotsc, x_n \in M$, is a form of `algebraic structure' on $M$. More precisely there is a $2$-monad $T = (T, \mu, \iota)$ on the $2$-category $\Cat$, of categories, functors and natural transformations, such that giving an \emph{action} $\map m{TM}M$, together with \emph{associator} and \emph{unitor} transformations $\bar m$ and $\hat m$ as below (both invertible and satisfying the usual coherence axioms), is the same as choosing finite products in $M$.
	\begin{displaymath}
		\begin{tikzpicture}[baseline]
			\matrix(m)[math4]{T^2M & TM \\ TM & M \\};
			\path[map]	(m-1-1) edge node[above] {$Tm$} (m-1-2)
													edge node[left] {$\mu_M$} (m-2-1)
									(m-1-2) edge node[right] {$m$} (m-2-2)
									(m-2-1) edge node[below] {$m$} (m-2-2);
			\path 			(m-1-2) edge[cell, below right, shorten >=9.5pt, shorten <=9.5pt] node {$\bar m$} (m-2-1);
		\end{tikzpicture} \qquad\qquad\qquad\qquad \begin{tikzpicture}[baseline]
			\matrix(m)[math4]{M & \\ TM & M \\};
			\path[map]	(m-1-1) edge node[left] {$\iota_M$} (m-2-1)
									(m-2-1) edge node[below] {$m$} (m-2-2);
			\path				(m-1-1) edge[eq, transform canvas={xshift=2pt, yshift=2pt}] (m-2-2)
									(m-1-2) edge[cell, shorten <=23.5pt, shorten >=2.5pt, transform canvas={xshift=-2pt}] node[xshift=-7pt, yshift=-5pt, below right] {$\hat m$} (m-2-1);
		\end{tikzpicture}
	\end{displaymath}
	
	In some more detail: the image of $M$ under $T$ has as objects (possibly empty) sequences $\ul x = (x_1, \dotsc, x_n)$ of objects of $M$ while, for another such sequence $\ul y = (y_1, \dotsc, y_k)$, its morphisms $\ul x \to \ul y$ are sequences $(s, u_1, \dotsc, u_k)$ consisting of a function $\map s{\brcs{1, \dotsc, k}}{\brcs{1, \dotsc, n}}$ and morphisms $\map{u_i}{x_{si}}{y_i}$ in $M$. The unit $\map{\iota_M}M{TM}$ inserts the objects and morphisms of $M$ into $TM$ as singleton sequences. In the correspondence mentioned above the action $\map m{TM}M$ maps $(x_1, \dotsc, x_n)$ to the chosen product $(x_1 \times \dotsb \times x_n) = m(x_1, \dotsc, x_n)$.
	
	The $2$-monad $T$ for categories with finite products is special in that it is `\emph{colax-idempotent}' in the sense of \cite{Kelly-Lack97}. In particular this means that a pseudo $T$-algebra structure $(m, \bar m, \hat m)$ on $M$, as above, is `essentially unique'. Precisely: there exists a cell $\cell{\theta_m}{\iota_M \of m}{\id_{TM}}$ such that the pair $(\hat m, \theta_m)$ defines $m$ as the right adjoint of $\map{\iota_M}M{TM}$. In our case the adjunction $\iota_M \ladj m$ is simply the universal property of products: maps $(x) \to (y_1, \dotsc, y_k)$ in $TM$ correspond to maps $x \to (y_1 \times \dotsb \times y_k)$ in $M$.
	
	Furthermore any morphism $\map fAC$ between pseudo $T$-algebras is uniquely a colax $T$-morphism, with structure cell $\cell{\bar f}{f \of a}{c \of Tf}$ given by the following composite (in anticipation of double categories we have already drawn the morphisms vertically and the cells as squares here). In our case $\bar f$ is the natural transformation consisting of the canonical maps $f(x_1 \times \dotsb \times x_n) \to (fx_1 \times \dotsb \times fx_n)$.
	\begin{equation} \label{colax algebra structure}
		\begin{tikzpicture}[textbaseline]
			\matrix(m)[math35]{
				TA & TA & TA & TA \\
				A & A & A & \\
				C & C & TA & TA \\
				& TC & TC & TC \\
				C & C & C & C \\ };
			\path[map]	(m-1-1) edge node[left] {$a$} (m-2-1)
									(m-1-2) edge node[left] {$a$} (m-2-2)
									(m-1-3) edge node[left] {$a$} (m-2-3)
									(m-2-1) edge node[left] {$f$} (m-3-1)
									(m-2-2) edge node[left] {$f$} (m-3-2)
									(m-2-3) edge node[left] {$\iota_A$} (m-3-3)
									(m-3-2) edge node[right] {$\iota_C$} (m-4-2)
									(m-3-3) edge node[right] {$Tf$} (m-4-3)
									(m-3-4) edge node[right] {$Tf$} (m-4-4)
									(m-4-2) edge node[right] {$c$} (m-5-2)
									(m-4-3) edge node[right] {$c$} (m-5-3)
									(m-4-4) edge node[right] {$c$} (m-5-4);
			\path				(m-1-1) edge[eq] (m-1-2)
									(m-1-2) edge[eq] (m-1-3)
									(m-1-3) edge[eq] (m-1-4)
									(m-1-4) edge[eq] (m-3-4)
									(m-2-1) edge[eq] (m-2-2)
									(m-2-2) edge[eq] (m-2-3)
									(m-3-1) edge[eq] (m-3-2)
													edge[eq] (m-5-1)
									(m-3-3) edge[eq] (m-3-4)
									(m-4-2) edge[eq] (m-4-3)
									(m-4-3) edge[eq] (m-4-4)
									(m-5-1) edge[eq] (m-5-2)
									(m-5-2) edge[eq] (m-5-3)
									(m-5-3) edge[eq] (m-5-4);
			\path[transform canvas={shift=($(m-3-3)!0.5!(m-4-3)$)}]	(m-3-1) edge[cell] node[right] {$\hat c$} (m-4-1)
									(m-1-3) edge[cell] node[right] {$\theta_a$} (m-2-3);
		\end{tikzpicture}
	\end{equation}
	
	In the above terms the first of the two statements that are asserted to be equivalent by \thmref{finite-product-preserving left Kan extensions} can be rephrased as follows. Writing $\map{l = \textup{lan}_{\yon_A} d}{\ps A}\Set$, as well as $\map v{T\ps A}{\ps A}$ and $\map w{T\Set}\Set$ for choices of finite products in $\ps A$ and $\Set$, it states that the unique structure cell $\cell{\bar l}{l \of v}{w \of Tl}$, that makes $l$ into a colax $T$-morphism, is invertible.
	
	\section{Double categories}
	We now briefly review the terminology of double categories. The notion of a \emph{double category} generalises that of a $2$-category by, instead of the usual single type, considering two types of morphism, one denoted $\map fAC$, $\map gBD, \dotsc$ and drawn vertically; the other denoted $\hmap JAB$, $\hmap KCD, \dotsc$ and drawn horizontally. The cells $\phi$, $\psi, \dotsc$ of a double category have both a vertical and horizontal morphism as source and as target, so that they are shaped like squares:
	\begin{displaymath}
	  \begin{tikzpicture}
	    \matrix(m)[math35]{A & B \\ C & D. \\};
	    \path[map]  (m-1-1) edge[barred] node[above] {$J$} (m-1-2)
	                        edge node[left] {$f$} (m-2-1)
	                (m-1-2) edge node[right] {$g$} (m-2-2)
	                (m-2-1) edge[barred] node[below] {$K$} (m-2-2);
	    \path[transform canvas={shift={($(m-1-2)!(0,0)!(m-2-2)$)}}] (m-1-1) edge[cell] node[right] {$\phi$} (m-2-1);
	  \end{tikzpicture}
	\end{displaymath}
	
	Vertical morphisms can be composed, as can horizontal morphisms, while cells can be composed both horizontally and vertically. As usual vertical composition and vertical identities, which are strictly associative and unital, are denoted by $\of$ and $\id$. Horizontal composition is denoted by $\hc$ and written in `diagrammatic order', while the horizontal unit $A \brar A$ is denoted $1_A$; often these are only associative and unital up to coherent invertible cells.  When drawing diagrams we depict both the horizontal and vertical units by the equal sign ($=$) while we leave identity cells empty, like we did in \eqref{colax algebra structure}. 
	
	Every double category $\K$ contains both a \emph{vertical $2$-category} $V(\K)$, consisting of its objects, vertical morphisms and \emph{vertical cells} (that is cells with units as horizontal source and target), as well as a \emph{horizontal bicategory} $H(\K)$, consisting of its objects, horizontal morphisms and \emph{horizontal cells} (cells with identities as vertical source and target). For further details we refer to \cite{Shulman08} by Shulman, which is very useful, or \cite{Koudenburg14a}
	
	\begin{example}
		The archetypical example of a double category is that of profunctors, denoted $\Prof$. Its objects are small categories and its vertical morphisms the functors between them, while its horizontal morphisms $\hmap JAB$ are profunctors; that is functors of the form $\map J{\op A \times B}\Set$. A cell $\phi$ as above is a natural transformation $\nat\phi J{K \of (\op f \times g)}$. Profunctors are composed by using coends (see Example~1.2 of \cite{Koudenburg14a}), while the unit profunctor $\hmap{1_A}AA$ is simply given by the hom-sets $1_A(x,y) = A(x,y)$. The vertical $2$-category $V(\Prof)$ contained in $\Prof$ is simply that of categories, functors and their transformations, while the horizontal bicategory $H(\Prof)$ is that of categories, profunctors and their transformations.
	
		The double category $\Prof$ has well-known generalisations $\enProf \V$ and $\inProf\E$, the first obtained by enriching over a suitable monoidal category $\V$ and the second by internalising in a suitable category $\E$ with pullbacks. We remark that in the former we do not require that $\V$ is closed symmetric monoidal: the horizontal morphisms $A \brar B$ of $\enProf\V$ are defined in a `bimodule-like way' (see Section 7 of \cite{DayStreet97}), and can be identified with the usual $\V$-functors $\op A \tens B \to \V$ in the case that $\V$ is closed symmetric monoidal.
	\end{example}
	
	\section{(Op-)cartesian cells}
	Invaluable to the theory of double categories are the notions of cartesian and opcartesian cell, which define restrictions and extensions of horizontal morphisms respectively, as follows. The cell $\phi$ above is called \emph{cartesian} if every cell $\chi$ below factors uniquely through $\phi$ as a cell $\psi$ as shown.
	\begin{displaymath}
		\begin{tikzpicture}[textbaseline]
    		\matrix(m)[math35]{X & Y \\ A & B \\ C & D \\};
    		\path[map]  (m-1-1) edge[barred] node[above] {$H$} (m-1-2)
        		                edge node[left] {$h$} (m-2-1)
        		        (m-1-2) edge node[right] {$k$} (m-2-2)
        		        (m-2-1) edge node[left] {$f$} (m-3-1)
        		        (m-2-2) edge node[right] {$g$} (m-3-2)
        		        (m-3-1) edge[barred] node[below] {$K$} (m-3-2);
    		\path[transform canvas={shift={($(m-2-1)!0.5!(m-1-1)$)}}] (m-2-2) edge[cell] node[right] {$\chi$} (m-3-2);
  		\end{tikzpicture} = \begin{tikzpicture}[textbaseline]
    		\matrix(m)[math35]{X & Y \\ A & B \\ C & D \\};
    		\path[map]  (m-1-1) edge[barred] node[above] {$H$} (m-1-2)
        		                edge node[left] {$h$} (m-2-1)
            		    (m-1-2) edge node[right] {$k$} (m-2-2)
            		    (m-2-1) edge[barred] node[below] {$J$} (m-2-2)
            		            edge node[left] {$f$} (m-3-1)
            		    (m-2-2) edge node[right] {$g$} (m-3-2)
            		    (m-3-1) edge[barred] node[below] {$K$} (m-3-2);
    		\path[transform canvas={shift=(m-2-1))}]
        		        (m-1-2) edge[cell] node[right] {$\psi$} (m-2-2)
        		        (m-2-2) edge[transform canvas={yshift=-0.3em}, cell] node[right] {$\phi$} (m-3-2);
  	\end{tikzpicture}
	\end{displaymath}
	If the cartesian cell $\phi$ exists then its horizontal source $J$ is called the \emph{restriction} of $K$ along $f$ and $g$, and denoted $K(f, g) = J$. By their universal property any two cartesian cells defining the same restriction factor through each other as invertible horizontal cells. In the double categories $\Prof$, $\enProf\V$ and $\inProf\E$ all restrictions exist; in $\Prof$ they are simply given by $K(f,g) = K \of (\op f \times g)$.
	
	Vertically dual, the cell $\psi$ in the right-hand side above is called \emph{opcartesian} if any cell $\chi$ factors uniquely through $\psi$ as a cell $\phi$, as shown. In that case the horizontal target $J$ of $\psi$ is called the \emph{extension} of $H$ along $h$ and $k$, which is again unique up to isomorphism. Cartesian cells satisfy the following pasting lemma; opcartesian cells satisfy its vertical dual.
	\begin{lemma}[Pasting lemma] If the cell $\phi$ in the right-hand side above is cartesian then $\chi$ is cartesian if and only if $\psi$ is.
	\end{lemma}
	
	For a vertical morphism $\map fAC$ the restrictions $\hmap{1_C(f, \id)}AC$ and $\hmap{1_C(\id, f)}CA$, of the unit $\hmap{1_C}CC$, are called the \emph{companion} and \emph{conjoint} of $f$; they are denoted $f_*$ and $f^*$ respectively. In the double category $\Prof$ of profunctors, the companion and the conjoint of a functor $\map fAC$ are simply the representable profunctors $f_*(x,y) = C(fx, y)$ and $f^*(y,x) = C(y, fx)$, where $x \in A$ and $y \in C$.
	
	While we have defined companions and conjoints as restrictions, the following lemma (and its horizontal dual) show that they can equivalently be defined as extensions along $f$. In fact it shows that, if the companion of $f$ exists, then the cartesian cells $\phi$ defining it correspond bijectively to the opcartesian cells $\psi$ defining it, in such a way that each corresponding pair $(\phi, \psi)$ satisfies the identities below, which we call the \emph{companion identities}. Analogous identities are satisfied by corresponding pairs of a cartesian and opcartesian cell defining the same conjoint; these we call the \emph{conjoint identities}.
	\begin{lemma} \label{unit cell factorisation}
		Consider a factorisation of the unit $1_f$ of $\map fAC$ as on the left below.   The following conditions are equivalent: $\phi$ is cartesian; $\psi$ is opcartesian; the identity on the right holds.
	 			\begin{displaymath}
	  	\begin{tikzpicture}[textbaseline]
	    	\matrix(m)[math35]{A & A \\ C & C \\};
	    	\path[map]  (m-1-1) edge node[left] {$f$} (m-2-1)
	    	            (m-1-2) edge node[right] {$f$} (m-2-2);
	    	\path				(m-1-1) edge[eq] (m-1-2)
	    							(m-2-1) edge[eq] (m-2-2);
	    	\path[transform canvas={xshift=1.75em}] (m-1-1) edge[cell] node[right] {$1_f$} (m-2-1);
	  	\end{tikzpicture} = \begin{tikzpicture}[textbaseline]
    		\matrix(m)[math35]{A & A \\ A & C \\ C & C \\};
    		\path[map]  (m-1-2) edge node[right] {$f$} (m-2-2)
            		    (m-2-1) edge[barred] node[below] {$J$} (m-2-2)
            		            edge node[left] {$f$} (m-3-1);
        \path				(m-1-1) edge[eq] (m-1-2)
        										edge[eq] (m-2-1)
        						(m-2-2) edge[eq] (m-3-2)
        						(m-3-1) edge[eq] (m-3-2);
    		\path[transform canvas={xshift=1.75em}]
        		        (m-1-1) edge[cell] node[right] {$\psi$} (m-2-1)
        		        (m-2-1) edge[transform canvas={yshift=-3pt}, cell] node[right] {$\phi$} (m-3-1);
  		\end{tikzpicture} \qquad\qquad \begin{tikzpicture}[textbaseline]
  			\matrix(m)[math35]{A & A & C \\ A & C & C \\};
  			\path[map]	(m-1-2) edge[barred] node[above] {$J$} (m-1-3)
  													edge node[right] {$f$} (m-2-2)
  									(m-2-1) edge[barred] node[below] {$J$} (m-2-2);
  			\path				(m-1-1) edge[eq] (m-1-2)
  													edge[eq] (m-2-1)
  									(m-1-3) edge[eq] (m-2-3)
  									(m-2-2) edge[eq] (m-2-3);
  			\path[transform canvas={xshift=1.75em}]	(m-1-1) edge[cell] node[right] {$\psi$} (m-2-1)
  									(m-1-2) edge[cell] node[right] {$\phi$} (m-2-2);
  		\end{tikzpicture} = \begin{tikzpicture}[textbaseline]
  			\matrix(m)[math35]{A & C \\ A & C \\};
  			\path[map]	(m-1-1) edge[barred] node[above] {$J$} (m-1-2)
  									(m-2-1) edge[barred] node[below] {$J$} (m-2-2);
  			\path				(m-1-1) edge[eq] (m-2-1)
  									(m-1-2) edge[eq] (m-2-2);
  			\path[transform canvas={xshift=1.75em, xshift=-5.5pt}]	(m-1-1) edge[cell] node[right] {$\id_J$} (m-2-1);
  		\end{tikzpicture}
	  \end{displaymath}
	\end{lemma}
	Thus a double category has all companions and conjoints whenever it has all restrictions or all extensions. Conversely restrictions and extensions can be build up from companions and conjoints as follows.  For a proof see Theorem 4.1 of \cite{Shulman08}.
	\begin{lemma} \label{cartesian and opcartesian cells in terms of companions and conjoints}
		Suppose that the companions and conjoints of the morphisms $\map fAC$ and $\map gBD$ exist. For any $\hmap KCD$ the composite on the left below is cartesian while, for any $\hmap JBA$, the composite on the right is opcartesian.
		\begin{displaymath}
			\begin{tikzpicture}[textbaseline]
				\matrix(m)[math35]{A & C & D & B \\ C & C & D & D \\};
				\path[map]	(m-1-1) edge[barred] node[above] {$f_*$} (m-1-2)
														edge node[left] {$f$} (m-2-1)
										(m-1-2) edge[barred] node[above] {$K$} (m-1-3)
										(m-1-3) edge[barred] node[above] {$g^*$} (m-1-4)
										(m-1-4) edge node[right] {$g$} (m-2-4)
										(m-2-2) edge[barred] node[below] {$K$} (m-2-3);
				\path				(m-1-2) edge[eq] (m-2-2)
										(m-1-3) edge[eq] (m-2-3)
										(m-2-1) edge[eq] (m-2-2)
										(m-2-3) edge[eq] (m-2-4);
				\draw				($(m-1-1)!0.5!(m-2-2)$) node[small] {\textup{cart}}
										($(m-1-3)!0.5!(m-2-4)$) node[small] {\textup{cart}};
			\end{tikzpicture} \qquad\qquad \begin{tikzpicture}[textbaseline]
				\matrix(m)[math35]{B & B & A & A \\ D & B & A & C \\};
				\path[map]	(m-1-1) edge node[left] {$g$} (m-2-1)
										(m-1-2) edge[barred] node[above] {$J$} (m-1-3)
										(m-2-3) edge[barred] node[below] {$f_*$} (m-2-4)
										(m-1-4) edge node[right] {$f$} (m-2-4)
										(m-2-1) edge[barred] node[below] {$g^*$} (m-2-2)
										(m-2-2) edge[barred] node[below] {$J$} (m-2-3);
				\path				(m-1-2) edge[eq] (m-2-2)
										(m-1-3) edge[eq] (m-2-3)
										(m-1-1) edge[eq] (m-1-2)
										(m-1-3) edge[eq] (m-1-4);
				\draw				($(m-1-1)!0.5!(m-2-2)$) node[small] {\textup{opcart}}
										($(m-1-3)!0.5!(m-2-4)$) node[small] {\textup{opcart}};
			\end{tikzpicture}
		\end{displaymath}
	\end{lemma}
	In summary the following conditions on a double category $\K$ are equivalent: $\K$ has all companions and conjoints; $\K$ has all restrictions; $\K$ has all extensions. Double categories satisfying these conditions, such as $\Prof$, $\enProf\V$ and $\inProf\E$, are called \emph{equipments}.
	
	Later we will use that,  in an equipment $\K$, choosing a companion $\hmap{f_*}AC$ for each morphism $\map fAC$ induces a pseudofunctor \mbox{$\map{(\dash)_*}{\co{V(\K)}}{H(\K)}$}, where $\co{V(\K)}$ is obtained from $V(\K)$ (the vertical $2$-category contained in $\K$) by reversing the directions of its cells, by mapping each vertical cell $\cell\phi fg$ of $\K$ to the horizontal cell $\cell{\phi_*}{g_*}{f_*}$ that is defined as the unique factorisation in
	\begin{equation} \label{companion pseudofunctor}
		\begin{tikzpicture}[textbaseline]
			\matrix[math35](m){A & A & C \\ C & C & C \\};
			\path[map]	(m-1-1) edge node[left] {$f$} (m-2-1)
									(m-1-2) edge[barred] node[above] {$g_*$} (m-1-3)
													edge node[right] {$g$} (m-2-2);
			\path				(m-1-1) edge[eq] (m-1-2)
									(m-1-3) edge[eq] (m-2-3)
									(m-2-1) edge[eq] (m-2-2)
									(m-2-2) edge[eq] (m-2-3);
			\path[map, transform canvas={xshift=1.75em}]	(m-1-1) edge[cell] node[right] {$\phi$} (m-2-1);
			\draw[transform canvas={shift={(1.75em, -1.625em)}}]	(m-1-2) node[small] {cart};
		\end{tikzpicture} = \begin{tikzpicture}[textbaseline]
			\matrix[math35](m){A & C \\ A & C \\ C & C. \\};
			\path[map]	(m-1-1) edge[barred] node[above] {$g_*$} (m-1-2)
									(m-2-1) edge[barred] node[below] {$f_*$} (m-2-2)
													edge node[left] {$f$} (m-3-1);
			\path				(m-1-1) edge[eq] (m-2-1)
									(m-1-2) edge[eq] (m-2-2)
									(m-2-2) edge[eq] (m-3-2)
									(m-3-1) edge[eq] (m-3-2);
			\path[map, transform canvas={xshift=1.75em}]	(m-1-1) edge[cell] node[right] {$\phi_*$} (m-2-1);
			\draw[transform canvas={shift={(1.75em, -1.625em)}}]	(m-2-1) node[small] {cart};
		\end{tikzpicture}
	\end{equation}
	For morphisms $\map fAC$ and $\map hCE$ the `compositor' $f_* \hc h_* \iso (h \of f)_*$ of this pseudofunctor is given as follows. Combining the previous lemma with the pasting lemma we find that the composite of the cartesian cells defining $f_*$ and $h_*$ is again cartesian, so that it defines $f_* \hc h_*$ as the companion of $h \of f$, and hence factors as an invertible horizontal cell $f_* \hc h_* \iso (h \of f)_*$ through the cartesian cell that defines the chosen companion of $h \of f$; these cells form the compositor of $(\dash)_*$.
	
	Analogously, choosing conjoints $\hmap{f^*}CA$ for each $\map fAC$ in $\K$ induces a pseudofunctor $\map{(\dash)^*}{\op{V(\K)}}{H(\K)}$, where $\op{V(\K)}$ is obtained from $V(\K)$ by reversing the directions of its morphisms. Notice that both $(\dash)_*$ and $(\dash)^*$ are locally full and faithful: for instance, the inverse to $\phi \mapsto \phi_*$ is given by mapping each horizontal cell $g_* \Rightarrow f_*$ to its composite with both the opcartesian cell that defines $g_*$ and the cartesian cell defining $f_*$. We record the following consequence for later use.
	\begin{lemma} \label{invertible companion cell}
		In a double category consider a vertical cell $\cell\phi fg$ between morphisms $\map fAC$ and $\map gAC$ that have companions $f_*$ and $g_*$. The cell $\phi$ is invertible precisely if the horizontal cell $\phi_*$, defined by the factorisation above, is invertible. 
	\end{lemma}
	
	The following simple consequence of \lemref{cartesian and opcartesian cells in terms of companions and conjoints} will be helpful as well.
	\begin{lemma} \label{horizontal composite of cartesian cells}
		Consider the following horizontal composite. If $f$ has a companion, $h$ has a conjoint and $\phi$ and $\psi$ are cartesian, then the composite itself is cartesian.
		\begin{displaymath}
			\begin{tikzpicture}
				\matrix(m)[math35]{A & B & E \\ C & B & F \\};
				\path[map]	(m-1-1) edge[barred] node[above] {$J$} (m-1-2)
														edge node[left] {$f$} (m-2-1)
										(m-1-2) edge[barred] node[above] {$H$} (m-1-3)
										(m-1-3) edge node[right] {$h$} (m-2-3)
										(m-2-1) edge[barred] node[below] {$K$} (m-2-2)
										(m-2-2) edge[barred] node[below] {$L$} (m-2-3);
				\path				(m-1-2) edge[eq] (m-2-2);
				\path[transform canvas={xshift=1.75em}]	(m-1-1) edge[cell] node[right] {$\phi$} (m-2-1)
										(m-1-2) edge[cell] node[right] {$\psi$} (m-2-2);
			\end{tikzpicture}
		\end{displaymath}
	\end{lemma}
	\begin{proof}
		By \lemref{cartesian and opcartesian cells in terms of companions and conjoints} the restriction of $K$ along $f$ can be given as the composite $f_* \hc K$ while that of $L$ along $h$ can be given as $L \hc h^*$. Thus there exist invertible horizontal cells  $J \iso f_* \hc K$ and $H \iso L \hc h^*$ obtained by factorising $\phi$ through the composite of the cartesian cell defining $f_*$ and $\id_K$, and $\psi$ through that of $\id_L$ and the cartesian cell defining $h^*$. We conclude that the composite $\phi \hc \psi$ itself factors as an invertible horizontal cell through the composite of $\id_{K \hc L}$ and the cartesian cells defining $f_*$ and $h^*$; since the latter composite is cartesian by \lemref{cartesian and opcartesian cells in terms of companions and conjoints}, so is the former.
	\end{proof}
	
	Closing this section, the next lemma reformulates the notion of adjunction in the vertical $2$-category $V(\K)$, in terms of the cartesian cells of $\K$.
	\begin{lemma} \label{unit of adjunction as cartesian cell}
		Consider morphisms $\map fAC$ and $\map gCA$ in a double category $\K$ and assume that the companion of $f$ exists. The cell $\phi$, on the left-hand side below, forms the unit defining an adjunction $f \ladj g$ in $V(\K)$ if and only if its factorisation $\phi'$, as shown, is cartesian in $\K$.
		\begin{displaymath}
	  	\begin{tikzpicture}[textbaseline]
    		\matrix(m)[math35]{A & A \\  & C \\ A & A \\};
    		\path[map]  (m-1-2) edge node[right] {$f$} (m-2-2)
        		        (m-2-2) edge node[right] {$g$} (m-3-2);
    		\path				(m-1-1) edge[eq] (m-1-2)
    												edge[eq] (m-3-1)
    								(m-3-1) edge[eq] (m-3-2);
    		\path[transform canvas={shift={($(m-2-1)!0.5!(m-1-1)$)}}] (m-2-2) edge[cell] node[right] {$\phi$} (m-3-2);
  		\end{tikzpicture} = \begin{tikzpicture}[textbaseline]
    		\matrix(m)[math35]{A & A \\ A & C \\ A & A \\};
    		\path[map]  (m-1-2) edge node[right] {$f$} (m-2-2)
            		    (m-2-1) edge[barred] node[below] {$f_*$} (m-2-2)
            		    (m-2-2) edge node[right] {$g$} (m-3-2);
        \path				(m-1-1) edge[eq] (m-1-2)
        										edge[eq] (m-2-1)
        						(m-2-1) edge[eq] (m-3-1)
        						(m-3-1) edge[eq] (m-3-2)
        						($(m-1-1)!0.5!(m-2-2)$) node[small] {\textup{opcart}};
    		\path[transform canvas={shift=(m-2-1))}]
        		        (m-2-2) edge[transform canvas={yshift=-0.3em}, cell] node[right] {$\phi'$} (m-3-2);
  		\end{tikzpicture}
	  \end{displaymath}
	\end{lemma}
	
	\section{Kan extensions}
	Next we turn to the notion of Kan extension in double categories, that was introduced in \cite{Grandis-Pare08}, as well as the stronger notion of pointwise Kan extension, introduced in \cite{Koudenburg14a}.
	
	A cell $\eta$ in a double category, as in the right-hand side below, is said to define $l$ as the \emph{left Kan extension} of $d$ along $J$ if any cell $\phi$ below factors uniquely through $\eta$ as a vertical cell $\psi$, as shown.
	\begin{displaymath}
  	\begin{tikzpicture}[textbaseline]
	    \matrix(m)[math35]{A & B \\ M & M \\};
	    \path[map]  (m-1-1) edge[barred] node[above] {$J$} (m-1-2)
	                        edge node[left] {$d$} (m-2-1)
	                (m-1-2) edge node[right] {$k$} (m-2-2);
	    \path				(m-2-1) edge[eq] (m-2-2);
	    \path[transform canvas={shift={($(m-1-2)!(0,0)!(m-2-2)$)}}] (m-1-1) edge[cell] node[right] {$\phi$} (m-2-1);
	  \end{tikzpicture} = \begin{tikzpicture}[textbaseline]
			\matrix(m)[math35]{A & B & B \\ M & M & M \\};
			\path[map]	(m-1-1) edge[barred] node[above] {$J$} (m-1-2)
													edge node[left] {$d$} (m-2-1)
									(m-1-2) edge node[right] {$l$} (m-2-2)
									(m-1-3) edge node[right] {$k$} (m-2-3);
			\path				(m-1-2) edge[eq] (m-1-3)
									(m-2-1) edge[eq] (m-2-2)
									(m-2-2) edge[eq] (m-2-3);
			\path[transform canvas={shift={($(m-1-2)!0.5!(m-2-3)$)}}] (m-1-1) edge[cell] node[right] {$\eta$} (m-2-1)
									(m-1-2) edge[cell] node[right] {$\psi$} (m-2-2);
		\end{tikzpicture}
  \end{displaymath}
  Using the universal property of opcartesian cells we see immediately that, in any double category $\K$, left Kan extensions along companions coincide with left Kan extensions in the $2$-category $V(\K)$; the latter in the classical sense (see e.g.\ Section~2.2 of \cite{Lack09}):
  \begin{lemma} \label{Kan extensions along companions}
  	Assume that the companion of $\map jAB$ exists in a double category $\K$. The vertical cell $\eta$ on the left-hand side below defines $l$ as the left Kan extension of $d$ along $j$ in $V(\K)$ precisely if its factorisation $\eta'$, as shown, defines $l$ as the left Kan extension of $d$ along $j_*$ in $\K$.
		\begin{displaymath}
			\begin{tikzpicture}[textbaseline]
    		\matrix(m)[math35]{A & A \\ \phantom B & B \\ M & M \\};
    		\path[map]  (m-1-1) edge node[left] {$d$} (m-3-1)
       			        (m-1-2) edge node[right] {$j$} (m-2-2)
       			        (m-2-2) edge node[right] {$l$} (m-3-2);
      	\path				(m-1-1) edge[eq] (m-1-2)
      							(m-3-1) edge[eq] (m-3-2);
    		\path[transform canvas={shift={($(m-2-1)!0.5!(m-1-1)$)}}] (m-2-2) edge[cell] node[right] {$\eta$} (m-3-2);
  		\end{tikzpicture} = \begin{tikzpicture}[textbaseline]
				\matrix(m)[math35]{A & A \\ A & B \\ M & M \\};
				\path[map]	(m-1-2) edge node[right] {$j$} (m-2-2)
										(m-2-1) edge[barred] node[below] {$j_*$} (m-2-2)
														edge node[left] {$d$} (m-3-1)
										(m-2-2) edge node[right] {$l$} (m-3-2);
				\path				(m-1-1) edge[eq] (m-1-2)
														edge[eq] (m-2-1)
										(m-3-1) edge[eq] (m-3-2);
				\path[transform canvas={shift=(m-2-1), yshift=-2pt}]	(m-2-2) edge[cell] node[right] {$\eta'$} (m-3-2);
				\draw				($(m-1-1)!0.5!(m-2-2)$)	node[small] {\textup{opcart}};
			\end{tikzpicture}
		\end{displaymath}
	\end{lemma}
	
	We say that the cell $\eta$ in the right-hand side below defines $l$ as the \emph{pointwise} left Kan extension of $d$ along $J$ if each cell $\phi$, this time of the more general form below, factors through $\eta$ as a unique cell $\psi$ as well, as shown.
	\begin{displaymath}
  	\begin{tikzpicture}[textbaseline]
			\matrix(m)[math35]{A & B & C \\ M & & M \\};
			\path[map]	(m-1-1) edge[barred] node[above] {$J$} (m-1-2)
													edge node[left] {$d$} (m-2-1)
									(m-1-2) edge[barred] node[above] {$H$} (m-1-3)
									(m-1-3) edge node[right] {$k$} (m-2-3)
									(m-1-2) edge[cell] node[right] {$\phi$} (m-2-2);
			\path				(m-2-1) edge[eq] (m-2-3);
		\end{tikzpicture} = \begin{tikzpicture}[textbaseline]
			\matrix(m)[math35]{A & B & C \\ M & M & M \\};
			\path[map]	(m-1-1) edge[barred] node[above] {$J$} (m-1-2)
													edge node[left] {$d$} (m-2-1)
									(m-1-2) edge[barred] node[above] {$H$} (m-1-3)
													edge node[right] {$l$} (m-2-2)
									(m-1-3) edge node[right] {$k$} (m-2-3);
			\path				(m-2-1) edge[eq] (m-2-2)
									(m-2-2) edge[eq] (m-2-3);
			\path[transform canvas={shift={($(m-1-2)!0.5!(m-2-3)$)}}] (m-1-1) edge[cell] node[right] {$\eta$} (m-2-1)
									(m-1-2) edge[cell] node[right] {$\psi$} (m-2-2);
		\end{tikzpicture}
  \end{displaymath}
  It is clear that every pointwise left Kan extension is in particular a left Kan extension. As usual any two cells defining the same (pointwise) left Kan extension factor through each other uniquely as invertible vertical cells.
  
  In a moment we will discuss how the above notion of pointwise left Kan extension generalises several similar and familiar notions. First we consider some of its properties. Cells defining pointwise left Kan extensions satisfy a pasting lemma, like (op-)cartesian cells do:
  \begin{lemma}[Pasting lemma]
  	Suppose that in the right-hand side above the cell $\eta$ defines $l$ as the pointwise left Kan extension of $d$ along $J$. Then $\phi$ defines $k$ as the (pointwise) left Kan extension of $d$ along $J \hc H$ precisely if $\psi$ defines $k$ as the (pointwise) left Kan extension of $l$ along $H$.
  \end{lemma}
  
  The following lemma is easily proved using the conjoint identities; see the discussion preceding \lemref{unit cell factorisation}.
  \begin{lemma} \label{conjoints define pointwise left Kan extensions}
  	Given a composable pair $\map fCB$ and $\map lBM$, assume that the conjoint $\hmap {f^*}BC$ of $f$ exists. The composite below defines $l \of f$ as the pointwise left Kan extension of $l$ along $f^*$.
  	\begin{displaymath}
  		\begin{tikzpicture}
  			\matrix(m)[math35]{B & C \\ B & B \\ M & M \\};
  			\path[map]	(m-1-1) edge[barred] node[above] {$f^*$} (m-1-2)
  									(m-1-2) edge node[right] {$f$} (m-2-2)
  									(m-2-1) edge node[left] {$l$} (m-3-1)
  									(m-2-2) edge node[right] {$l$} (m-3-2);
  			\path				(m-1-1) edge[eq] (m-2-1)
  									(m-2-1) edge[eq] (m-2-2)
  									(m-3-1) edge[eq] (m-3-2);
  			\draw[transform canvas={shift={(1.75em, -1.625em)}}]	(m-1-1) node[small] {\textup{cart}};
  		\end{tikzpicture}
  	\end{displaymath}
  \end{lemma} 
  
  Pointwise left Kan extensions deserve their name because of the following lemma, which says that `any restriction of a pointwise left Kan extension is again a pointwise left Kan extension'. For a proof use that $J(\id, f) \iso J \hc f^*$ by \lemref{cartesian and opcartesian cells in terms of companions and conjoints} and the previous two lemmas.
  \begin{lemma}\label{restriction of pointwise left Kan extensions}
  	Assume that the cell $\eta$ in the composite below defines $l$ as the pointwise left Kan extension of $d$ along $J$ and let $\map fCB$ be a morphism whose conjoint $\hmap{f^*}BC$ exists. The full composite defines $l \of f$ as the pointwise left Kan extension of $d$ along the restriction $J(\id, f)$.
  	\begin{displaymath}
  		\begin{tikzpicture}
				\matrix(m)[math35]{A & C \\ A & B \\ M & M \\};
				\path[map]	(m-1-1) edge[barred] node[above] {$J(\id, f)$} (m-1-2)
										(m-1-2)	edge node[right] {$f$} (m-2-2)
										(m-2-1) edge[barred] node[below] {$J$} (m-2-2)
														edge node[left] {$d$} (m-3-1)
										(m-2-2) edge node[right] {$l$} (m-3-2);
				\path				(m-1-1) edge[eq] (m-2-1)
										(m-3-1) edge[eq] (m-3-2);
				\path[transform canvas={shift=(m-2-1), yshift=-2pt}]	(m-2-2) edge[cell] node[right] {$\eta$} (m-3-2);
				\draw				($(m-1-1)!0.5!(m-2-2)$)	node[small] {\textup{cart}};
			\end{tikzpicture}
  	\end{displaymath}
  \end{lemma}
  
  A vertical morphism $\map fAC$ is called \emph{full and faithful} if its unit cell $1_f$ is cartesian. In the equipments $\Prof$ and $\enProf\V$ of ($\V$-enriched) profunctors this definition coincides with the usual notion. In an equipment notice that $f$ is full and faithful precisely if the opcartesian cell defining $f_*$ is cartesian, by applying the pasting lemma for cartesian cells to the factorisation of $1_f$ that was considered in \lemref{unit cell factorisation}. Invoking the previous lemma we obtain the following variation of a classical result.
  \begin{lemma} \label{pointwise left Kan extension along full and faithful map}
  	Suppose that the cell $\eta$ in the composite below defines $l$ as the pointwise left Kan extension of $d$ along the companion of $\map jAB$. If $j$ is full and faithful then the full composite is invertible.
  	\begin{displaymath} 	
  	 	\begin{tikzpicture}
				\matrix(m)[math35]{A & A \\ A & B \\ M & M \\};
				\path[map]	(m-1-2) edge node[right] {$j$} (m-2-2)
										(m-2-1) edge[barred] node[below] {$j_*$} (m-2-2)
														edge node[left] {$d$} (m-3-1)
										(m-2-2) edge node[right] {$l$} (m-3-2);
				\path				(m-1-1) edge[eq] (m-1-2)
														edge[eq] (m-2-1)
										(m-3-1) edge[eq] (m-3-2);
				\path[transform canvas={shift=(m-2-1), yshift=-2pt}]	(m-2-2) edge[cell] node[right] {$\eta$} (m-3-2);
				\draw				($(m-1-1)!0.5!(m-2-2)$)	node[small] {\textup{opcart}};
			\end{tikzpicture}
		\end{displaymath}
  \end{lemma}
  
  The following is also well-known.
  \begin{lemma} \label{adjunction with invertible unit}
  	In a double category $\K$ consider a vertical cell $\cell\phi{\id_A}{g \of f}$, that forms the unit of an adjunction $\map{f \ladj g}CA$, and assume that the companion of $f$ exists. The cell $\phi$ is invertible if and only if $f$ is full and faithful.
  \end{lemma}
  \begin{proof}
  	Consider the factorisation $\phi = \phi' \of \textup{opcart}$, as in \lemref{unit of adjunction as cartesian cell}, where $\phi'$ is cartesian and thus defines $g$ as the pointwise left Kan extension of $\id_A$ along $f_*$, by \lemref{conjoints define pointwise left Kan extensions}. The `if'-part follows immediately by applying the previous lemma to this factorisation.
  	
  	For the `only if'-part: if $\phi$ is invertible then it is cartesian. Since $\phi'$ is cartesian too, it follows from the pasting lemma that the opcartesian cell defining $f_*$ is cartesian as well, which is equivalent to $f$ being full and faithful.
  \end{proof}
  
  As hinted at by \lemref{restriction of pointwise left Kan extensions} the notion of pointwise left Kan extension in equipments is closely related to familiar other such notions:
  \begin{itemize}
  	\item In equipments $\K$ that have so-called `opcartesian tabulations', such as the equipment $\inProf\E$ of profunctors internal to $\E$, pointwise left Kan extensions along companions coincide with pointwise left Kan extensions in the $2$-category $V(\K)$---the latter in the sense of Street's paper \cite{Street74}. For details see \cite{Koudenburg14a}.
  	\item Similarly, for a cocomplete closed symmetric monoidal category $\V$, pointwise left Kan extensions along companions in the equipment $\enProf\V$, of $\V$-enriched profunctors, coincide with left Kan extensions along enriched functors in the classical sense, as defined in e.g.\ Kelly's book \cite{Kelly82}. In fact, we can regard the notion of pointwise left Kan extension in $\enProf\V$ as an extension of this classical notion, to the cases in which $\V$ is not closed symmetric monoidal but merely monoidal; see Section 1.5 of \cite{Koudenburg14b}.
  	\item In another direction, when considered in a `closed' equipment our notion of pointwise left Kan extension is closely related to Wood's notion of `indexed colimit', that was introduced in \cite{Wood82}.
  \end{itemize}

  \section{The Beck-Chevalley condition}
  Now that we have most of the necessary terminology of double categories in place, we can begin reformulating \thmref{finite-product-preserving left Kan extensions}. We start with its condition on the copresheaf $\map dA\Set$ and, in the example below, restate it as a `Beck-Chevalley condition'.
  
  First we notice that the $2$-monad $T$ for categories with finite products, on the $2$-category $\Cat = V(\Prof)$, extends to a `double monad' on $\Prof$, which we again denote $T$. Without going into every detail, this extension maps a profunctor $\hmap JAB$ to the profunctor $\hmap{TJ}{TA}{TB}$ that is given by
  \begin{displaymath}
  	(TJ)\bigpars{(x_1, \dotsc, x_n),(y_1, \dotsc, y_k)} = \coprod_{\map s{\set{1, \dotsc, k}}{\set{1, \dotsc, n}}} \prod_{i=1}^k J(x_{si},y_i);
  \end{displaymath}
  compare the definition of the hom-sets of $TA$ given before. The `double transformation' that is the extension of the unit transformation $\nat\iota{\id_\Cat}T$ consists, besides the usual natural functors $\map{\iota_A}A{TA}$, of natural cells $\cell{\iota_J}J{TJ}$ in $\Prof$ as on the left below. These are simply given by the natural isomorphisms $J(x,y) \iso (TJ)(\iota x, \iota y)$, for all $x \in A$ and $y \in B$.
  \begin{equation} \label{right Beck-Chevalley composite}
	  \begin{tikzpicture}[textbaseline]
	    \matrix(m)[math35]{A & B \\ TA & TB \\};
	    \path[map]  (m-1-1) edge[barred] node[above] {$J$} (m-1-2)
	                        edge node[left] {$\iota_A$} (m-2-1)
	                (m-1-2) edge node[right] {$\iota_B$} (m-2-2)
	                (m-2-1) edge[barred] node[below] {$TJ$} (m-2-2);
	    \path[transform canvas={shift={($(m-1-2)!(0,0)!(m-2-2)$)}}] (m-1-1) edge[cell] node[right] {$\iota_J$} (m-2-1);
	  \end{tikzpicture} \qquad\qquad \iota_{J*} = \begin{tikzpicture}[textbaseline]
	  	\matrix(m)[math35]{A & A & B & TB \\ A & TA & TB & TB \\};
	  	\path[map]	(m-1-2) edge[barred] node[above] {$J$} (m-1-3)
	  											edge node[right] {$\iota_A$} (m-2-2)
	  							(m-1-3) edge[barred] node[above] {$\iota_{B*}$} (m-1-4)
	  											edge node[right, inner sep=1pt] {$\iota_B$} (m-2-3)
	  							(m-2-1) edge[barred] node[below] {$\iota_{A*}$} (m-2-2)
	  							(m-2-2) edge[barred] node[below] {$TJ$} (m-2-3);
	  	\path				(m-1-1) edge[eq] (m-1-2)
	  											edge[eq] (m-2-1)
	  							(m-1-4) edge[eq] (m-2-4)
	  							(m-2-3) edge[eq] (m-2-4);
	  	\path[transform canvas={shift={($(m-1-3)!0.5!(m-2-3)$)}, xshift=3pt}]	(m-1-2) edge[cell] node[right] {$\iota_J$} (m-2-2);
	  	\draw				($(m-1-1)!0.5!(m-2-2)$) node[small] {\opcart}
	  							($(m-1-3)!0.5!(m-2-4)$) node[small, xshift=3pt] {\cart};
	  \end{tikzpicture}
	\end{equation}
	For each profunctor $\hmap JAB$ we write $\iota_{J*}$ for the composite on the right above; compare the assignment $\phi \mapsto \phi_*$ defined by \eqref{companion pseudofunctor}. We say that $\iota_J$ satisfies the \emph{right Beck-Chevalley condition} if $\iota_{J*}$ is invertible.
	
	\begin{example} \label{right Beck-Chevalley condition}
		Returning to the copresheaf $\map dA\Set$ of \thmref{finite-product-preserving left Kan extensions}, we regard $\Set$ as the presheaf category $\ps 1 = \brks{1, \Set} \iso \Set$, where $1$ is the terminal category, and write $\hmap{D = \ps 1(\yon_1, d)}1A$ for the restriction of the unit profunctor on $\ps 1$ along the yoneda embedding $\map{\yon_1}1{\ps 1}$ and $\map dA{\ps 1}$; clearly $D(*,x) \iso dx$ for all $x \in A$. Later we will generalise the assignment $d \mapsto D$ thus given to other equipments that admit `yoneda embeddings'. Returning to the situation of \thmref{finite-product-preserving left Kan extensions}, we claim that $\iota_D$ satisfies the right Beck-Chevalley condition precisely if the category of elements $\int d$ is cosifted.
		
		Indeed, after remembering that the source of the transformation $\cell{\iota_{D*}}{D \hc \iota_{A*}}{\iota_{1*} \hc TD}$ is defined using coends and that its target is isomorphic to the restriction $(TD)(\iota_1, \id)$ (by \lemref{cartesian and opcartesian cells in terms of companions and conjoints}) it is a simple exercise to show that
		\begin{itemize}
			\item the component of $\iota_{D*}$ at $(*, ()) \in 1 \times TA$ is a bijection precisely if $\el d$ is connected;
			\item for any $k \geq 1$ and $\ul y = (y_1, \dotsc, y_k) \in TA$, the component of $\iota_{D*}$ at $(*, \ul y)$ is a bijection precisely if the category $\Span{(y_1, q_1), \dotsc, (y_k, q_k)}$ of `$k$-spans' in $\el d$ is connected for every sequence $(q_1, \dotsc, q_k)$ of elements $q_i \in dy_i$.
		\end{itemize}
		It is not hard to see that the second condition restricted to the case of $k = 2$ implies all cases $k \geq 1$, and we conclude that $\iota_D$ satisfies the right Beck-Chevalley condition if and only if $\el d$ is cosifted.
	\end{example}
	
	Before we can state properties of the right Beck-Chevalley condition we have to make more concrete the notion of `double monad'. What follows are brief descriptions of the notions of double functor, double transformation and double monad; for further details see for instance Section 2 of \cite{Koudenburg14b}.
	
	Given double categories $\K$ and $\L$, a \emph{lax double functor} $\map F\K\L$ is given by assignments that map the objects, vertical morphisms, horizontal morphisms and cells of $\K$ to those of $\L$ in a way that is compatible with sources and targets, and that preserves vertical composition and identities. Horizontal composition and units need only be preserved up to coherent horizontal cells $\cell{F_\hc}{FJ \hc FH}{F(J \hc H)}$ and $\cell{F_1}{1_{FA}}{F1_A}$, called the \emph{compositor} and \emph{unitor} cells, which satisfy associativity and unit axioms. We will often use the unit axioms, which state that the composite
	\begin{displaymath}
		1_{FA} \hc FJ \xRar{F_1 \hc \id} F1_A \hc FJ \xRar{F_\hc} F(1_A \hc FJ) \xRar{F\mf l} FJ
	\end{displaymath}
	equals $\mf l$ for each $\hmap JAB$, and likewise for $\mf r$. The double functor $F$ is called \emph{normal} if its unitor cells are identities, while it is called \emph{pseudo} if both its compositor and unitor cells are invertible.
	
	Given lax double functors $F$ and $\map G\K\L$, a \emph{double transformation} $\nat\xi FG$ consists of a family of morphisms $\map{\xi_A}{FA}{GA}$, indexed by $A \in \K$, and a family of cells $\xi_J$ as below, indexed by $\hmap JAB \in \K$. Both families are required to be natural with respect to vertical composition, that to satisfy a composition and unit axiom. If $F$ and $G$ are normal lax double functors then the unit axiom simply states that $\xi_{1_A} = 1_{\xi_A}$.
	\begin{displaymath}
		\begin{tikzpicture}
	    \matrix(m)[math35]{FA & FB \\ GA & GB \\};
	    \path[map]  (m-1-1) edge[barred] node[above] {$FJ$} (m-1-2)
	                        edge node[left] {$\xi_A$} (m-2-1)
	                (m-1-2) edge node[right] {$\xi_B$} (m-2-2)
	                (m-2-1) edge[barred] node[below] {$GJ$} (m-2-2);
	    \path[transform canvas={shift={($(m-1-2)!(0,0)!(m-2-2)$)}}] (m-1-1) edge[cell] node[right] {$\xi_J$} (m-2-1);
	  \end{tikzpicture}
	\end{displaymath}
	
	Finally a \emph{lax double monad} $T = (T, \mu, \iota)$ on a double category $\K$ consists of a lax double functor $\map T\K\K$ and double transformations $\nat\mu{T^2}T$ and $\nat\iota{\id_\K}T$, the latter satisfying the usual axioms $\mu \of \mu T = \mu \of T\mu$ and $\mu \of \iota T = \id_T = \mu \of T\iota$. For convenience we will restrict ourselves to double monads whose underlying lax double functor $T$ is normal.
	
	Before looking at the right Beck-Chevalley condition, we first record some facts about double functors and double transformations.
	\begin{lemma} \label{double functors preserve companions and conjoints}
		Normal lax double functors preserve cartesian and opcartesian cells that define companions and conjoints.
	\end{lemma}
	For a proof notice that images of the cells $\phi$ and $\psi$, as in \lemref{unit cell factorisation}, again satisfy the companion identities (for the horizontal identity use the unit axiom for double functors). Extending the previous lemma, the following is half of Proposition~6.8 of \cite{Shulman08}, which is harder to prove.
	\begin{proposition}[Shulman] \label{lax double functors preserve cartesian cells}
		Lax double functors preserve all cartesian cells.
	\end{proposition}
	
	By combining \lemref{double functors preserve companions and conjoints} with the discussion following \eqref{companion pseudofunctor} the first assertion of the next lemma follows. In the proof of the second assertion use the unit axiom for $F$.
	\begin{lemma} \label{compositor on companions}
		Consider morphisms $\map fAC$ and $\map hCE$ in a double category $\K$ whose companions $\hmap{f_*}AC$ and $\hmap{h_*}CE$ exist, and let $\map F\K\L$ be a normal lax double functor. Both the composite of the $F$-images of the cartesian cells defining $f_*$ and $h_*$, as well as the $F$-image of the composite of these cells, define the companion of $Fh \of Ff$ in $\L$. The unique factorisation of the former through the latter is the compositor $\cell{F_\hc}{Ff_* \hc Fh_*}{F(f_* \hc h_*)}$, which is therefore invertible.
	\end{lemma}
	
	\begin{lemma} \label{naturality of double transformations}
		Let $\nat\xi FG$ be a double transformation between normal lax double functors $F$ and $\map G\K\L$. For the companion of a morphism $\map fAC$ in $\K$ the following identity holds; a similar one holds for its conjoint $\hmap{f^*}CA$.
		\begin{displaymath}
			\begin{tikzpicture}[textbaseline]
				\matrix(m)[math35]{FA & FC \\ GA & GC \\ GC & GC \\};
				\path[map]	(m-1-1) edge[barred] node[above] {$Ff_*$} (m-1-2)
														edge node[left] {$\xi_A$} (m-2-1)
										(m-1-2) edge node[right] {$\xi_C$} (m-2-2)
										(m-2-1) edge[barred] node[below] {$Gf_*$} (m-2-2)
														edge node[left] {$Gf$} (m-3-1);
				\path				(m-2-2) edge[eq] (m-3-2)
										(m-3-1) edge[eq] (m-3-2);
				\path[transform canvas={xshift=1.75em}]	(m-1-1) edge[cell] node[right] {$\xi_{f_*}$} (m-2-1);
				\draw[transform canvas={shift={(1.75em, -1.625em)}}]	(m-2-1) node[small] {$G$\textup{cart}};
			\end{tikzpicture} = \begin{tikzpicture}[textbaseline]
				\matrix(m)[math35]{FA & FA & FC \\ GA & FC & FC \\ GC & GC & GC \\};
				\path[map]	(m-1-1) edge node[left] {$\xi_A$} (m-2-1)
										(m-1-2) edge[barred] node[above] {$Ff_*$} (m-1-3)
														edge node[left] {$Ff$} (m-2-2)
										(m-2-1) edge node[left] {$Gf$} (m-3-1)
										(m-2-2) edge node[right] {$\xi_C$} (m-3-2)
										(m-2-3) edge node[right] {$\xi_C$} (m-3-3);
				\path				(m-1-1) edge[eq] (m-1-2)
										(m-1-3) edge[eq] (m-2-3)
										(m-2-2) edge[eq] (m-2-3)
										(m-3-1) edge[eq] (m-3-2)
										(m-3-2) edge[eq] (m-3-3);
				\path[transform canvas={shift={(1.75em, -1.625em)}}]	(m-1-2) node[small] {$F$\textup{cart}};
			\end{tikzpicture}
		\end{displaymath}
	\end{lemma}
	\begin{proof}
		The proof follows from the naturality of $\nat\xi FG$ and its unit axiom.
	\end{proof}
	
	For the remainder of this section we fix a normal lax double monad $T = (T, \mu, \iota)$ on an equipment $\K$, and turn to the right Beck-Chevalley condition that was defined at the beginning of this section.
	\begin{proposition} \label{right Beck-Chevalley condition for companions}
		The cell $\iota_{f_*}$ satisfies the right Beck-Chevalley condition for each $\map fAC$. 
	\end{proposition}
	\begin{proof}
		We have show to that the horizontal cell $(\iota_{f_*})_*$ in the left-hand side below is invertible. To this end consider the full equation below, whose first identity follows from the definition of $(\iota_{f_*})_*$ and the companion identity for $\iota_{A*}$, while the second follows from the previous lemma.
		\begin{displaymath}
			\begin{tikzpicture}[textbaseline]
				\matrix[math35](m){A & C & TC \\ A & TA & TC \\ TA & TA & TC \\ TC & TC & TC \\};
				\path[map]	(m-1-1) edge[barred] node[above] {$f_*$} (m-1-2)
										(m-1-2) edge[barred] node[above] {$\iota_{C*}$} (m-1-3)
										(m-2-1) edge[barred] node[below] {$\iota_{A*}$} (m-2-2)
														edge node[left] {$\iota_A$} (m-3-1)
										(m-2-2) edge[barred] node[below] {$Tf_*$} (m-2-3)
										(m-3-1) edge node[left] {$Tf$} (m-4-1)
										(m-3-2) edge[barred] node[above] {$Tf_*$} (m-3-3)
														edge node[left] {$Tf$} (m-4-2);
				\path				(m-1-1) edge[eq] (m-2-1)
										(m-1-3) edge[eq] (m-2-3)
										(m-2-2) edge[eq] (m-3-2)
										(m-2-3) edge[eq] (m-3-3)
										(m-3-1) edge[eq] (m-3-2)
										(m-3-3) edge[eq] (m-4-3)
										(m-4-1) edge[eq] (m-4-2)
										(m-4-2) edge[eq] (m-4-3)
										(m-1-2) edge[cell] node[right] {$(\iota_{f_*})_*$} (m-2-2);
				\draw[transform canvas={shift={(1.75em, -1.625em)}}]	(m-2-1) node[small] {\cart}
										(m-3-2) node[small] {$T$\cart};
			\end{tikzpicture} \! = \! \begin{tikzpicture}[textbaseline]
				\matrix[math35](m){A & C & TC \\ TA & TC & TC \\ TC & TC & TC \\};
				\path[map]	(m-1-1) edge[barred] node[above] {$f_*$} (m-1-2)
														edge node[left] {$\iota_A$} (m-2-1)
										(m-1-2) edge[barred] node[above] {$\iota_{C*}$} (m-1-3)
														edge node[right, inner sep=1pt] {$\iota_C$} (m-2-2)
										(m-2-1) edge[barred] node[below] {$Tf_*$} (m-2-2)
														edge node[left] {$Tf$} (m-3-1);
				\path				(m-1-3) edge[eq] (m-2-3)
										(m-2-2) edge[eq] (m-2-3)
														edge[eq] (m-3-2)
										(m-2-3) edge[eq] (m-3-3)
										(m-3-1) edge[eq] (m-3-2)
										(m-3-2) edge[eq] (m-3-3);
				\path[transform canvas={xshift=1.75em}]	(m-1-1) edge[cell] node[right] {$\iota_{f_*}$} (m-2-1);
				\draw[transform canvas={shift={(1.75em, -1.625em)}}] (m-1-2) node[small, xshift=3pt] {\cart}
										(m-2-1) node[small] {$T$\cart};
			\end{tikzpicture} \! = \! \begin{tikzpicture}[textbaseline]
				\matrix[math35](m){A & A & C & TC \\ TA & C & C & TC \\ TC & TC & TC & TC \\};
				\path[map]	(m-1-1) edge node[left] {$\iota_A$} (m-2-1)
										(m-1-2) edge[barred] node[above] {$f_*$} (m-1-3)
														edge node[left] {$f$} (m-2-2)
										(m-1-3) edge[barred] node[above] {$\iota_{C*}$} (m-1-4)
										(m-2-1) edge node[left] {$Tf$} (m-3-1)
										(m-2-2) edge node[left] {$\iota_C$} (m-3-2)
										(m-2-3) edge[barred] node[above] {$\iota_{C*}$} (m-2-4)
														edge node[left] {$\iota_C$} (m-3-3);
				\path				(m-1-1) edge[eq] (m-1-2)
										(m-1-3) edge[eq] (m-2-3)
										(m-1-4) edge[eq] (m-2-4)
										(m-2-2) edge[eq] (m-2-3)
										(m-2-4) edge[eq] (m-3-4)
										(m-3-1) edge[eq] (m-3-2)
										(m-3-2) edge[eq] (m-3-3)
										(m-3-3) edge[eq] (m-3-4);
				\draw[transform canvas={shift={(1.75em, -1.625em)}}]	(m-1-2) node[small] {\cart}
										(m-2-3) node[small] {\cart};
			\end{tikzpicture}
		\end{displaymath}
		Now notice that the bottom two rows of the left-hand side, as well as the last two columns of the right-hand side, are cartesian; this follows from \propref{lax double functors preserve cartesian cells} and the discussion following \eqref{companion pseudofunctor}. Applying \lemref{invertible companion cell} we conclude that $(\iota_{f_*})_*$ is invertible.
	\end{proof}
	
	\begin{corollary} \label{right Beck-Chevalley condition for conjoints of right adjoints}
		If $\map gCA$ has a left adjoint then $\iota_{g^*}$ satisfies the right Beck\ndash Che\-val\-ley condition.
	\end{corollary}
	\begin{proof}
		Consider a vertical cell $\cell\phi{\id_A}{g \of f}$ that forms the unit of an adjunction $f \ladj g$. It factors as a cartesian cell $\phi'$ that defines $f_*$ as the conjoint of $g$ by \lemref{unit of adjunction as cartesian cell} so that, by further factorising through the cartesian cell that defines $g^*$ as the conjoint of $g$, we obtain an invertible horizontal cell $\cell{\phi''}{f_*}{g^*}$. The naturality of $\iota$ implies that the diagram
		\begin{displaymath}
			\begin{tikzpicture}
				\matrix[math4, column sep=3em](m){f_* \hc \iota_{C*} & g^* \hc \iota_{C*} \\ \iota_{A*} \hc Tf_* & \iota_{A*} \hc Tg^* \\};
				\path				(m-1-1) edge[cell] node[above] {$\phi'' \hc \id$} (m-1-2)
														edge[cell] node[left] {$(\iota_{f_*})_*$} (m-2-1)
										(m-1-2) edge[cell] node[right] {$(\iota_{g^*})_*$} (m-2-2)
										(m-2-1) edge[cell] node[below] {$\id \hc T\phi''$} (m-2-2);
			\end{tikzpicture}
		\end{displaymath}
		commutes. As $(\iota_{f_*})_*$ is invertible by the previous proposition we conclude that, besides $(\iota_{g^*})_*$, all these cells are invertible. It follows that $(\iota_{g^*})_*$ must be invertible as well.
	\end{proof}
	
	The following result shows that, in the case that the double monad $T$ restricts to a colax-idempotent $2$-monad on the vertical $2$-category $V(\K)$ that is contained in $\K$, the right Beck-Chevalley condition for conjoints of maps $A \to C$ can be regarded as an extension of the notion of pseudo $T$\ndash morphisms $A \to C$. Here the latter only makes sense if both $A$ and $C$ admit algebra structures $(a, \bar a, \hat a)$ and $(c, \bar c, \hat c)$; recall that in that case any $\map fAC$ can be uniquely equipped with a structure cell $\cell{\bar f}{f \of a}{c \of Tf}$ that makes it into a colax $T$-morphism (see \eqref{colax algebra structure}).
	\begin{proposition} \label{right Beck-Chevalley condition for colax T-morphisms}
		Assume that the double monad $T$ restricts to a colax-idempotent $2$-monad on $V(\K)$ and let $\map fAC$ be a morphism between pseudo $T$-algebras. The unique structure cell $\bar f$, that makes $f$ into a colax $T$-morphism, is invertible precisely if the cell $\iota_{f^*}$ satisfies the right Beck-Chevalley condition.
	\end{proposition}
	\begin{proof}
		Let the algebra structures on $A$ and $C$ be given by the triples $(a, \bar a, \hat a)$ and $(c, \bar c, \hat c)$. Recall that the structure cell $\cell{\bar f}{f \of a}{c \of Tf}$ is given in terms of the unit $\cell{\hat c}{c \of \iota_C}{\id_C}$  of $C$ and the counit $\cell{\theta_a}{\id_{TA}}{\iota_A \of a}$ of the adjunction $\iota_A \ladj a$. Consider the equality below, where $\hat c'$ and $\hat a'$ are the factorisations of $\hat c$ and $\hat a$ through the opcartesian cells defining the companions of $\iota_C$ and $\iota_A$ respectively; its identities are explained below.
		\begin{multline*}
			\begin{tikzpicture}[textbaseline, ampersand replacement=\&]
				\matrix[math35](m){C \& A \& TA \\ C \& TC \& TA \\ C \& TC \& TC \\ C \& C \& C \\};
				\path[map]	(m-1-1) edge[barred] node[above] {$f^*$} (m-1-2)
										(m-1-2) edge[barred] node[above] {$\iota_{A*}$} (m-1-3)
										(m-2-1) edge[barred] node[below] {$\iota_{C*}$} (m-2-2)
										(m-2-2) edge[barred] node[below] {$Tf^*$} (m-2-3)
										(m-2-3) edge node[right] {$Tf$} (m-3-3)
										(m-3-1) edge[barred] node[above] {$\iota_{C*}$} (m-3-2)
										(m-3-2) edge node[right] {$\iota_C$} (m-4-2)
										(m-3-3) edge node[right] {$\iota_C$} (m-4-3);
				\path				(m-1-1) edge[eq] (m-2-1)
										(m-1-3) edge[eq] (m-2-3)
										(m-2-1) edge[eq] (m-3-1)
										(m-2-2) edge[eq] (m-3-2)
										(m-3-1) edge[eq] (m-4-1)
										(m-3-2) edge[eq] (m-3-3)
										(m-4-1) edge[eq] (m-4-2)
										(m-4-2) edge[eq] (m-4-3)
										(m-1-2) edge[cell] node[right] {$(\iota_{f^*})_*$} (m-2-2);
				\path[transform canvas={xshift=1.75em}]	(m-3-1) edge[cell] node[right] {$\hat c'$} (m-4-1);
				\draw[transform canvas={shift={(1.75em, -1.625em)}}]	(m-2-2) node[small] {$T$cart};
			\end{tikzpicture} = \begin{tikzpicture}[textbaseline, ampersand replacement=\&]
				\matrix[math35](m){C \& A \& A \& TA \\ C \& C \& TA \& TA \\ \& TC \& TC \& TC \\ C \& C \& C \& C \\};
				\path[map]	(m-1-1) edge[barred] node[above] {$f^*$} (m-1-2)
										(m-1-2) edge node[right] {$f$} (m-2-2)
										(m-1-3) edge[barred] node[above] {$\iota_{A*}$} (m-1-4)
														edge node[right, inner sep=1pt] {$\iota_A$} (m-2-3)
										(m-2-2) edge node[right] {$\iota_C$} (m-3-2)
										(m-2-3) edge node[right] {$Tf$} (m-3-3)
										(m-2-4) edge node[right] {$Tf$} (m-3-4)
										(m-3-2) edge node[right] {$c$} (m-4-2)
										(m-3-3) edge node[right] {$c$} (m-4-3)
										(m-3-4) edge node[right] {$c$} (m-4-4);
				\path				(m-1-1) edge[eq] (m-2-1)
										(m-1-2) edge[eq] (m-1-3)
										(m-1-4) edge[eq] (m-2-4)
										(m-2-1) edge[eq] (m-2-2)
														edge[eq] (m-4-1)
										(m-2-3) edge[eq] (m-2-4)
										(m-3-2) edge[eq] (m-3-3)
										(m-3-3) edge[eq] (m-3-4)
										(m-4-1) edge[eq] (m-4-2)
										(m-4-2) edge[eq] (m-4-3)
										(m-4-3) edge[eq] (m-4-4);
				\path[transform canvas={shift={(1.75em, -1.625em)}}]	(m-2-1) edge[cell] node[right] {$\hat c$} (m-3-1);
				\draw[transform canvas={shift={(1.75em, -1.625em)}}]	(m-1-1) node[small] {cart}
										(m-1-3) node[small] {cart};
			\end{tikzpicture} \\
			= \begin{tikzpicture}[textbaseline, ampersand replacement=\&]
				\matrix[math35](m){C \& A \& TA \& TA \\ C \& A \& A \& \\ C \& C \& TA \& TA \\ \& TC \& TC \& TC \\ C \& C \& C \& C \\};
				\path[map]	(m-1-1) edge[barred] node[above] {$f^*$} (m-1-2)
										(m-1-2) edge[barred] node[above] {$\iota_{A*}$} (m-1-3)
										(m-1-3) edge node[right] {$a$} (m-2-3)
										(m-2-1) edge[barred] node[above] {$f^*$} (m-2-2)
										(m-2-2) edge node[right] {$f$} (m-3-2)
										(m-2-3) edge node[right] {$\iota_A$} (m-3-3)
										(m-3-2) edge node[right] {$\iota_C$} (m-4-2)
										(m-3-3) edge node[right] {$Tf$} (m-4-3)
										(m-3-4) edge node[right] {$Tf$} (m-4-4)
										(m-4-2) edge node[right] {$c$} (m-5-2)
										(m-4-3) edge node[right] {$c$} (m-5-3)
										(m-4-4) edge node[right] {$c$} (m-5-4);
				\path				(m-1-1) edge[eq] (m-2-1)
										(m-1-2) edge[eq] (m-2-2)
										(m-1-3) edge[eq] (m-1-4)
										(m-1-4) edge[eq] (m-3-4)
										(m-2-1) edge[eq] (m-3-1)
										(m-2-2) edge[eq] (m-2-3)
										(m-3-1) edge[eq] (m-3-2)
														edge[eq] (m-5-1)
										(m-3-3) edge[eq] (m-3-4)
										(m-4-2) edge[eq] (m-4-3)
										(m-4-3) edge[eq] (m-4-4)
										(m-4-2) edge[eq] (m-4-3)
										(m-4-3) edge[eq] (m-4-4)
										(m-5-1) edge[eq] (m-5-2)
										(m-5-2) edge[eq] (m-5-3)
										(m-5-3) edge[eq] (m-5-4);
				\path[transform canvas={xshift=1.75em}]	(m-1-2) edge[cell] node[right] {$\hat a'$} (m-2-2);
				\path[transform canvas={shift={(1.75em, -1.625em)}}]	(m-1-3) edge[cell] node[right] {$\theta_a$} (m-2-3)
										(m-3-1) edge[cell] node[right] {$\hat c$} (m-4-1);
				\draw[transform canvas={shift={(1.75em, -1.625em)}}]	(m-2-1) node[small] {cart};
			\end{tikzpicture} = \begin{tikzpicture}[textbaseline, ampersand replacement=\&]
				\matrix[math35](m){C \& A \& TA \& TA \\ C \& A \& A \& TC \\ C \& C \& C \& C \\};
				\path[map]	(m-1-1) edge[barred] node[above] {$f^*$} (m-1-2)
										(m-1-2) edge[barred] node[above] {$\iota_{A*}$} (m-1-3)
										(m-1-3) edge node[right] {$a$} (m-2-3)
										(m-1-4) edge node[right] {$Tf$} (m-2-4)
										(m-2-1) edge[barred] node[above] {$f^*$} (m-2-2)
										(m-2-2) edge node[right] {$f$} (m-3-2)
										(m-2-3) edge node[right] {$f$} (m-3-3)
										(m-2-4) edge node[right] {$c$} (m-3-4);
				\path				(m-1-1) edge[eq] (m-2-1)
										(m-1-2) edge[eq] (m-2-2)
										(m-1-3) edge[eq] (m-1-4)
										(m-2-1) edge[eq] (m-3-1)
										(m-2-2) edge[eq] (m-2-3)
										(m-3-1) edge[eq] (m-3-2)
										(m-3-2) edge[eq] (m-3-3)
										(m-3-3) edge[eq] (m-3-4);
				\path[transform canvas={xshift=1.75em}]	(m-1-2) edge[cell] node[right] {$\hat a'$} (m-2-2);
				\path[transform canvas={shift={(1.75em, -1.625em)}}]	(m-1-3) edge[cell] node[right] {$\bar f$} (m-2-3);
				\draw[transform canvas={shift={(1.75em, -1.625em)}}]	(m-2-1) node[small] {cart};
			\end{tikzpicture}
		\end{multline*}
		
		The first identity here follows from the definitions of $(\iota_{f^*})_*$ and $\hat c'$, together with \lemref{naturality of double transformations}. The second follows from the fact that the composite of $\hat a'$ and $\theta$ coincides with the cartesian cell defining $\iota_{A*}$ (a consequence of one of the triangle identities for $\hat a$ and $\theta_a$, together with the uniqueness of factorisations through the opcartesian cell defining $\iota_{A*}$). The third identity follows from the definition of $\bar f$. Finally remember that $\hat c'$ and $\hat a'$ are cartesian by \lemref{unit of adjunction as cartesian cell}, so that the bottom two rows in the left-hand side above, as well as the first two columns in the right-hand side, are cartesian. The proof now follows from (the horizontal dual of) \lemref{invertible companion cell}.
	\end{proof}
	Finally we consider the right Beck-Chevalley condition in relation to restrictions of the form $K(\id, g)$.
	\begin{lemma} \label{right Beck-Chevalley condition for restrictions}
		For morphisms $\hmap KCD$ and $\map gBD$, the cell $\iota_{K(\id, g)}$ satisfies the right Beck\ndash Chevalley condition as soon as the cells $\iota_K$ and $\iota_{g^*}$ do.
	\end{lemma}
	\begin{proof}
		Notice that, for any restriction $K(\id, g)$ of $K$ along $g$, we have $K(\id, g) \iso K \hc g^*$ by \lemref{cartesian and opcartesian cells in terms of companions and conjoints} so that, using the naturality of the double transformation $\nat\iota{\id_\K}T$ like we did in the proof of \lemref{right Beck-Chevalley condition for conjoints of right adjoints}, we may assume that $K(\id, g) = K \hc g^*$ without loss of generality.
		
		Consider the identity below, where $\cell{T_\hc}{TK \hc Tg^*}{T(K \hc g^*)}$ is the compositor of $T$, and where we have denoted by $\gamma$ the cartesian cell that defines the conjoint $g^*$. That the identity holds follows from the composition axiom for  $\iota$ (see Section 2 of \cite{Koudenburg14b}).
		\begin{displaymath}
			\begin{tikzpicture}[textbaseline]
				\matrix[math35](m){C & C & D & B & TB \\ C & TC & & TB & TB \\ C & TC && TD & TD \\};
				\path[map]	(m-1-2) edge[barred] node[above] {$K$} (m-1-3)
														edge node[right] {$\iota_C$} (m-2-2)
										(m-1-3) edge[barred] node[above] {$g^*$} (m-1-4)
										(m-1-4) edge[barred] node[above] {$\iota_{B*}$} (m-1-5)
														edge node[left] {$\iota_B$} (m-2-4)
										(m-2-1) edge[barred] node[below] {$\iota_{C*}$} (m-2-2)
										(m-2-2) edge[barred] node[below, inner sep=2pt] {$T(K \hc g^*)$} (m-2-4)
										(m-2-4) edge node[right] {$Tg$} (m-3-4)
										(m-2-5) edge node[right] {$Tg$} (m-3-5)
										(m-3-1) edge[barred] node[below] {$\iota_{C*}$} (m-3-2)
										(m-3-2) edge[barred] node[below] {$TK$} (m-3-4);
				\path				(m-1-1) edge[eq] (m-1-2)
														edge[eq] (m-2-1)
										(m-1-5) edge[eq] (m-2-5)
										(m-2-1) edge[eq] (m-3-1)
										(m-2-2) edge[eq] (m-3-2)
										(m-2-4) edge[eq] (m-2-5)
										(m-3-4) edge[eq] (m-3-5)
										(m-1-3) edge[cell, transform canvas={xshift=-11pt}] node[right] {$\iota_{K \hc g^*}$} (m-2-3)
										(m-2-3) edge[cell, transform canvas={shift={(-18pt,-2pt)}}] node[right] {$T(\id_K \hc \gamma)$} (m-3-3);
				\draw[transform canvas={shift={(1.75em,-1.625em)}}]	(m-1-1) node[small] {\opcart}
										(m-1-4) node[small] {\cart};
			\end{tikzpicture} = \mspace{-9mu}\begin{tikzpicture}[textbaseline]
				\matrix[math35](m){C & C & D & B & TB \\ C & TC & TD & TB & TB \\ C & TC & & TB & TB \\ C & TC & & TD & TD \\};
				\path[map]	(m-1-2) edge[barred] node[above] {$K$} (m-1-3)
														edge node[right, inner sep=1pt] {$\iota_C$} (m-2-2)
										(m-1-3) edge[barred] node[above] {$g^*$} (m-1-4)
														edge node[right, inner sep=1pt] {$\iota_D$} (m-2-3)
										(m-1-4) edge[barred] node[above] {$\iota_{B*}$} (m-1-5)
														edge node[right, inner sep=1pt] {$\iota_B$} (m-2-4)
										(m-2-1) edge[barred] node[below] {$\iota_{C*}$} (m-2-2)
										(m-2-2) edge[barred] node[below] {$TK$} (m-2-3)
										(m-2-3) edge[barred] node[below] {$Tg^*$} (m-2-4)
										(m-3-1) edge[barred] node[below] {$\iota_{C*}$} (m-3-2)
										(m-3-2) edge[barred] node[below] {$T(K \hc g^*)$} (m-3-4)
										(m-3-4) edge node[right] {$Tg$} (m-4-4)
										(m-3-5) edge node[right] {$Tg$} (m-4-5)
										(m-4-1) edge[barred] node[below] {$\iota_{C*}$} (m-4-2)
										(m-4-2) edge[barred] node[below] {$TK$} (m-4-4);
				\path				(m-1-1) edge[eq] (m-1-2)
														edge[eq] (m-2-1)
										(m-1-5) edge[eq] (m-2-5)
										(m-2-1) edge[eq] (m-3-1)
										(m-2-2) edge[eq] (m-3-2)
										(m-2-4) edge[eq] (m-3-4)
														edge[eq] (m-2-5)
										(m-2-5) edge[eq] (m-3-5)
										(m-3-1) edge[eq] (m-4-1)
										(m-3-2) edge[eq] (m-4-2)
										(m-3-4) edge[eq] (m-3-5)
										(m-4-4) edge[eq] (m-4-5)
										(m-2-3) edge[cell] node[right] {$T_\hc$} (m-3-3)
										(m-3-3) edge[cell, transform canvas={shift={(-18pt, -2pt)}}] node[right] {$T(\id_K \hc \gamma)$} (m-4-3);
				\path[transform canvas={xshift=1.75em}]	(m-1-2) edge[cell] node[right] {$\iota_K$} (m-2-2)
										(m-1-3) edge[cell] node[right] {$\iota_{g^*}$} (m-2-3);
				\draw[transform canvas={shift={(1.75em,-1.625em)}}]	(m-1-1) node[small] {\opcart}
										(m-1-4) node[small, xshift=3pt] {\cart};
			\end{tikzpicture}
		\end{displaymath}
		We claim that the bottom row of the left-hand side as well as the bottom two rows of the right-hand side are cartesian. Indeed, this follows from  \lemref{cartesian and opcartesian cells in terms of companions and conjoints}, the fact that $T$ preserves cartesian cells and that, in the right-hand side, $T(\id_K \hc \gamma) \of T_\hc = \id_{TK} \hc T\gamma$ by the naturality of $T_\hc$ and its unit axiom (see Section 2 of \cite{Koudenburg14b}).
		
		We conclude that the top row of left-hand side, which equals $(\iota_{K \hc g^*})_*$, is invertible precisely if that of the right-hand side is. Using the companion identity for $\iota_{D*}$ we see that the latter coincides with the composite of $\iota_{K*}$ and $(\iota_{g^*})_*$, from which the lemma follows.
	\end{proof}
	
	\section{Yoneda embeddings}
	Finally we need a notion of `yoneda embedding' in equipments. The following definition is an adaptation of axioms satisfied by the morphisms that make up a `yoneda structure' on a $2$-category, in the sense of \cite{Weber07}.
	\begin{definition} \label{yoneda embedding}
		A morphism $\map{\yon_A}A{\ps A}$ in a double category is called a \emph{yoneda embedding} if it satisfies the following axioms:
		\begin{itemize}
			\item[(c)] for every horizontal morphism $\hmap JAB$ there exists a cartesian cell as on the left below;
			\item[(e)] if a cell $\eta$, as on the right, is cartesian then it defines $l$ as the pointwise left Kan extension of $\yon_A$ along $J$.
		\end{itemize}
		\begin{displaymath}
			\begin{tikzpicture}[baseline]
		    \matrix(m)[math35]{A & B \\ \ps A & \ps A \\};
		    \path[map]  (m-1-1) edge[barred] node[above] {$J$} (m-1-2)
		                        edge[ps] node[left] {$\yon_A$} (m-2-1)
		                (m-1-2) edge[ps] node[right] {$g$} (m-2-2);
		    \path				(m-2-1) edge[eq] (m-2-2);
		    \draw				($(m-1-1)!0.5!(m-2-2)$) node[small] {cart};
		  \end{tikzpicture} \qquad\qquad\qquad\qquad \begin{tikzpicture}[baseline]
	    	\matrix(m)[math35]{A & B \\ \ps A & \ps A \\};
	    	\path[map]  (m-1-1) edge[barred] node[above] {$J$} (m-1-2)
	    	                    edge[ps] node[left] {$\yon_A$} (m-2-1)
	    	            (m-1-2) edge[ps] node[right] {$l$} (m-2-2);
	    	\path				(m-2-1) edge[eq] (m-2-2);
	    	\path[transform canvas={shift={($(m-1-2)!(0,0)!(m-2-2)$)}}] (m-1-1) edge[cell] node[right] {$\eta$} (m-2-1);
	  	\end{tikzpicture}
	  \end{displaymath}
	\end{definition}
	
	Before showing that in the equipment $\enProf\V$ the usual yoneda embeddings satisfy the definition above, we consider some conditions that are equivalent to its condition (e). Recall that every double category $\K$ contains a horizontal bicategory $H(\K)$ consisting of its objects, its horizontal morphisms and its horizontal cells.
	
	\begin{lemma} \label{equivalent axioms}
		Consider a morphism $\map{\yon_A}A{\ps A}$ in a double category $\K$. If its companion $\hmap{\yon_{A*}}A{\ps A}$ exists then axiom (e) of the definition above is equivalent to condition (a) below. Additionally, if $\K$ is an equipment then both (e) and (a) are equivalent to condition (b).
		\begin{itemize}
			\item[(a)] 	The cartesian cell below, that defines the companion of $\yon_A$, defines $\id_{\ps A}$ as the pointwise left Kan extension of $\yon_A$ along $\yon_{A*}$;
			\begin{displaymath}
				\begin{tikzpicture}
		  	  \matrix(m)[math35]{A & \ps A \\ \ps A & \ps A \\};
		  	  \path[map]  (m-1-1) edge[barred] node[above] {$\yon_{A*}$} (m-1-2)
		  	                      edge[ps] node[left] {$\yon_A$} (m-2-1);
		  	  \path       (m-1-2) edge[eq, ps] (m-2-2)
		  	  						(m-2-1) edge[eq] (m-2-2);
		  	  \draw				($(m-1-1)!0.5!(m-2-2)$) node[small] {\textup{cart}};
		  	\end{tikzpicture}
	    \end{displaymath}
			\item[(b)]	for any morphisms $\hmap H{\ps A}C$ and $\map kC{\ps A}$ the assignment
	  	\begin{displaymath}
	  		\map{\yon_{A*} \hc \dash}{H(\K)(\ps A, C)(H, k^*)}{H(\K)(A, C)(\yon_{A*} \hc H, \yon_{A*} \hc k^*)},
	  	\end{displaymath}
	  	that is given by whiskering with $\yon_{A*}$, is bijective.
	  \end{itemize}
	\end{lemma}
	\begin{proof}
		(e) $\Rightarrow$ (a) is clear.
		
		(a) $\Rightarrow$ (e). Given a cartesian cell $\eta$ as on the left below, we consider its factorisation $\eta'$ as shown; it is cartesian because $\eta$ is, by the pasting lemma. The implication follows from applying \lemref{restriction of pointwise left Kan extensions}.
		\begin{displaymath}
			\begin{tikzpicture}[textbaseline]
	    	\matrix(m)[math35]{A & B \\ \ps A & \ps A \\};
	    	\path[map]  (m-1-1) edge[barred] node[above] {$J$} (m-1-2)
	    	                    edge[ps] node[left] {$\yon_A$} (m-2-1)
	    	            (m-1-2) edge[ps] node[right] {$l$} (m-2-2);
	    	\path				(m-2-1) edge[eq] (m-2-2);
	    	\path[transform canvas={shift={($(m-1-2)!(0,0)!(m-2-2)$)}}] (m-1-1) edge[cell] node[right] {$\eta$} (m-2-1);
	  	\end{tikzpicture} = \mspace{-12mu}\begin{tikzpicture}[textbaseline]
	  		\matrix(m)[math35]{A & B \\ A & \ps A \\ \ps A & \ps A \\};
	  		\path[map]	(m-1-1) edge[barred] node[above] {$J$} (m-1-2)
	  								(m-1-2) edge[ps] node[right] {$l$} (m-2-2)
	  								(m-2-1) edge[barred] node[below] {$\yon_{A*}$} (m-2-2)
	  												edge[ps] node[left] {$\yon_A$} (m-3-1);
	  		\path				(m-1-1) edge[eq] (m-2-1)
	  								(m-2-2) edge[eq, ps] (m-3-2)
	  								(m-3-1) edge[eq] (m-3-2);
	  		\path[transform canvas={xshift=1.75em}]	(m-1-1) edge[cell] node[right] {$\eta'$} (m-2-1);
	  		\draw[transform canvas={shift={(1.75em, -1.625em)}}] (m-2-1) node[small] {cart};
	  	\end{tikzpicture} \mspace{24mu} \begin{tikzpicture}[textbaseline]
	  		\matrix(m)[math35]{A & \ps A & C \\ \ps A & & \ps A \\};
	  		\path[map]	(m-1-1) edge[barred] node[above] {$\yon_{A*}$} (m-1-2)
	  												edge[ps] node[left] {$\yon_A$} (m-2-1)
	  								(m-1-2) edge[barred] node[above] {$H$} (m-1-3)
	  								(m-1-3) edge[ps] node[right] {$k$} (m-2-3);
	  		\path				(m-2-1) edge[eq] (m-2-3)
	  								(m-1-2) edge[cell] node[right] {$\phi$} (m-2-2);						
	  	\end{tikzpicture} \mspace{24mu} \begin{tikzpicture}[textbaseline]
	  		\matrix(m)[math35]{A & \ps A & C \\ \ps A & \ps A & \ps A \\};
	  		\path[map]	(m-1-1) edge[barred] node[above] {$\yon_{A*}$} (m-1-2)
	  												edge[ps] node[left] {$\yon_A$} (m-2-1)
	  								(m-1-2) edge[barred] node[above] {$k^*$} (m-1-3)
	  								(m-1-3) edge[ps] node[right] {$k$} (m-2-3);
	  		\path				(m-1-2) edge[ps, eq] (m-2-2)
	  								(m-2-1) edge[eq] (m-2-2)
	  								(m-2-2) edge[eq] (m-2-3);
	  		\draw[transform canvas={shift={(1.75em, -1.625em)}}]	(m-1-1) node[small] {cart}
	  								(m-1-2) node[small] {cart};
	  	\end{tikzpicture}
		\end{displaymath}
		
		(a) $\Leftrightarrow$ (b). Notice that condition (a) is the statement that any cell $\phi$, as in the middle above, factors uniquely through the cartesian cell defining $\yon_{A*}$, while condition (b) states that any horizontal cell $\yon_{A*} \hc H \Rightarrow \yon_{A*} \hc k^*$ factors uniquely through the identity cell $\id_{\yon_{A*}}$. Consider the composite of cartesian cells on the right above; it is itself cartesian by \lemref{cartesian and opcartesian cells in terms of companions and conjoints}. The equivalence of (a) and (b) follows from the fact that, under the unique factorisations of cells of the form $\phi$ through the cartesian cell on the right, the unique factorisations of (a) correspond precisely to those of (b).
	\end{proof}
	
	\begin{example} \label{enriched yoneda embedding}
		Let $\V = (\V, \tens, \brks{\dash, \dash})$ be a closed symmetric monoidal category that is both complete and cocomplete. In the equipment $\enProf\V$ of $\V$-profunctors the usual yoneda embedding $\map{\yon_A}A{\ps A = \brks{\op A, \V}}$, that maps $x \in A$ to the representable $\V$-presheaf $\yon_A x = A(\dash, x)$, satisfies the axioms of \defref{yoneda embedding}. For axiom (c): given a $\V$-profunctor $\hmap JAB$ we can take $\map gB{\ps A}$ to be given by $gy = J(\dash ,y)$ and the cartesian cell to consist of the yoneda isomorphisms $J(x,y) \iso \ps A(\yon_A x, gy)$.
		
		To see that condition (b) of the previouse lemma holds, we notice that the functor
		\begin{displaymath}
			\map{\yon_{A*} \hc \dash}{H(\enProf\V)(\ps A, B)}{H(\enProf\V)(A, B)}
		\end{displaymath}
		has a right adjoint $E$ that maps $\hmap JAB$ to the $\V$-profunctor given by $(EJ)(p, y) = \ps A(p, J(\dash, y))$. The components $\yon_{A*} \hc EJ \Rightarrow J$ of the counit of this adjunction are given by the isomorphisms $(\yon_{A*} \hc EJ)(x,y) \iso \ps A(\yon_A x, J(\dash, y)) \iso J(x,y)$. Those $\nat{\eta_H}H{E(\yon_{A*} \hc H)}$ of the unit, where $\hmap H{\ps A}B$, are induced by the family of maps $H(p, y) \to \brks{px, (\yon_{A*} \hc H)(x,y)}$ that are adjoint to
		\begin{displaymath}
			H(p, y) \tens px \iso H(p,y) \tens \ps A(\yon_A x, p) \xrar{\textup{act}} H(\yon_A x, y) \iso (\yon_{A*} \hc H)(x,y).
		\end{displaymath}
		That condition (b) holds now follows from the fact that $\eta_{g^*}$ is invertible for each $\map gB{\ps A}$, which is not hard to check.
	\end{example}
	
	We close this section with the following variation of Lemma 3.2 of \cite{Weber07}. Notice that it implies that, like in the classical case, any yoneda embedding \mbox{$\map{\yon_A}A{\ps A}$} is full and faithful, by applying it to the unit cell $1_{\yon_A}$ which trivially defines $\yon_A$ as the pointwise left Kan extension of $\yon_A$ along $1_A$.
	\begin{lemma} \label{left Kan extensions along yoneda embeddings}
		Let $\map{\yon_A}A{\ps A}$ be a yoneda embedding in a double category. If the cell $\eta$ below defines $l$ as the left Kan extension of $J$ along $\yon_A$ then it is cartesian and, moreover, it defines $l$ as a pointwise left Kan extension.
		\begin{displaymath}
			\begin{tikzpicture}
	    	\matrix(m)[math35]{A & B \\ \ps A & \ps A \\};
	    	\path[map]  (m-1-1) edge[barred] node[above] {$J$} (m-1-2)
	    	                    edge[ps] node[left] {$\yon_A$} (m-2-1)
	    	            (m-1-2) edge[ps] node[right] {$l$} (m-2-2);
	    	\path				(m-2-1) edge[eq] (m-2-2);
	    	\path[transform canvas={shift={($(m-1-2)!(0,0)!(m-2-2)$)}}] (m-1-1) edge[cell] node[right] {$\eta$} (m-2-1);
	  	\end{tikzpicture}
	  \end{displaymath}
	\end{lemma}
	\begin{proof}
		By axiom (a) of \defref{yoneda embedding} there exists a morphism $\map gB{\ps A}$ such that $J$ is the restriction of $1_{\ps A}$ along $\yon_A$ and $g$, and the cartesian cell $\eps$ defining $J$ as this restriction defines $g$ as the pointwise left Kan extension of $\yon_A$ along $J$ by axiom (e). Since $\eta$ defines the same left Kan extension we conclude that it factors through $\eps$ as an invertible vertical cell $g \iso l$; it follows that, like $\eps$, $\eta$ is cartesian and defines $l$ as a pointwise left Kan extension.
	\end{proof}
	
	\section{The results}
	Now we are ready to formulate and prove a generalisation of \thmref{finite-product-preserving left Kan extensions}, as well as some related results. Throughout this section we consider a normal lax double monad $T = (T, \mu, \iota)$ on an equipment $\K$, and assume that it restricts to a colax-idempotent $2$-monad (again denoted $T$) on $V(\K)$.
	
	\subsection{In the case that \texorpdfstring{$d$}{d} is of the form \texorpdfstring{$\map dA\ps M$}{d: A → M̂}}
	We start with the result that generalises \thmref{finite-product-preserving left Kan extensions}. It concerns pointwise left Kan extensions along maps that are of the form $\map dA\ps M$, where $\ps M$ is the target of a yoneda embedding $\map{\yon_M}M\ps M$. As in \exref{right Beck-Chevalley condition} we write $D = \ps M(\yon_M, d)$ for the restriction that is defined by the cartesian cell on the left below.
	\begin{displaymath}
		\begin{tikzpicture}[baseline]
		  \matrix(m)[math35]{ M & A \\ \ps M & \ps M \\};
		  \path[map]  (m-1-1) edge[barred] node[above] {$D$} (m-1-2)
		                      edge[ps] node[left] {$\yon_M$} (m-2-1)
		              (m-1-2) edge[ps] node[right] {$d$} (m-2-2);
		  \path				(m-2-1) edge[eq] (m-2-2);
		  \draw				($(m-1-1)!0.5!(m-2-2)$) node[small] {cart};
		\end{tikzpicture} \qquad\qquad\qquad \begin{tikzpicture}[baseline]
	    \matrix(m)[math35]{A & B \\ \ps M & \ps M \\};
	    \path[map]  (m-1-1) edge[barred] node[above] {$j_*$} (m-1-2)
	                        edge[ps] node[left] {$d$} (m-2-1)
	                (m-1-2) edge[ps] node[right] {$l$} (m-2-2);
	    \path				(m-2-1) edge[eq] (m-2-2);
	    \path[transform canvas={shift={($(m-1-2)!(0,0)!(m-2-2)$)}}] (m-1-1) edge[cell] node[right] {$\eta$} (m-2-1);
	  \end{tikzpicture} \qquad\qquad\qquad \begin{tikzpicture}[textbaseline]
			\matrix(m)[math35, column sep=1.5em]{TA & TB \\ T\ps M & T\ps M \\ \ps M & \ps M \\};
			\path[map]	(m-1-1) edge[barred] node[above] {$Tj_*$} (m-1-2)
													edge[ps] node[left] {$Td$} (m-2-1)
									(m-1-2) edge[ps] node[right] {$Tl$} (m-2-2)
									(m-2-1) edge[ps] node[left] {$w$} (m-3-1)
									(m-2-2) edge[ps] node[right] {$w$} (m-3-2);
			\path				(m-2-1) edge[eq] (m-2-2)
									(m-3-1) edge[eq] (m-3-2);
			\path[transform canvas={shift=(m-2-2)}]	(m-1-1) edge[cell] node[right] {$T\eta$} (m-2-1);
		\end{tikzpicture}
	\end{displaymath}
	Next we generalise the situation of \thmref{finite-product-preserving left Kan extensions}: suppose that the cell $\eta$ in the middle above defines $\map lB{\ps M}$ as the pointwise left Kan extension of $d$ along the companion of $\map jAB$, and that $B$ and $\ps M$ admit pseudo $T$-algebra structures $(b, \bar b, \hat b)$ and $(w, \bar w, \hat w)$ respectively. Remember that, $T$ being colax-idempotent, this means that $\iota_B \ladj b$ and $\iota_{\ps M} \ladj w$, and that $\map lB{\ps M}$ admits a unique colax $T$\ndash morphism structure $\cell{\bar l}{l \of b}{w \of Tl}$, that is given as a composite of the unitor $\hat w$ of $\ps M$ and the counit $\theta_b$ of the adjunction $\iota_B \ladj b$, like in \eqref{colax algebra structure}.
	
	\begin{theorem} \label{thm1}
		In the above situation the structure cell $\bar l$ is invertible whenever the cell $\iota_D$ satisfies the right Beck-Chevalley condition and the composite $1_w \of T\eta$, on the right above, defines $w \of Tl$ as the left Kan extension of $w \of Td$ along $Tj_*$. The converse holds as soon as $\map jAB$ is full and faithful.
	\end{theorem}
	
	In our case, where $T$ is the double monad of categories with finite products, the composite $1_w \of T\eta$ on the right above defines $w \of Tl$ as a left Kan extension as soon as taking finite products $(x_1, \dotsc, x_n) \mapsto (x_1 \times \dotsb \times x_n)$ in $\ps M$ preserves colimits in each variable. This happens in particular when $\ps M$ is closed cartesian, e.g.\ when $\ps M = \brks{\op M, \Set}$. We conclude that, by taking $M = 1$ and by taking \mbox{$\map{\yon_M}1{\brks{1, \Set} \iso \Set}$} and $j = \map{\yon_A}A{\brks{\op A, \Set}}$ to be the usual yoneda embeddings, the theorem above reduces to the assertion that $\bar l$ is invertible precisely if $\iota_D$ satisfies the right Beck-Chevalley condition. This, as follows from our discussions in Section 1 and \exref{right Beck-Chevalley condition}, is a reformulation of \thmref{finite-product-preserving left Kan extensions}.
	
	\begin{proof}[Proof of \thmref{thm1}]
		To start we denote by $\zeta$ the factorisation of the cartesian cell that defines the restriction $D = \ps M(\yon_M, d)$, through the cartesian cell defining $\yon_{M*}$, as shown below. Notice that $\zeta$ is cartesian by the pasting lemma and, since lax double functors preserve cartesian cells (see \propref{lax double functors preserve cartesian cells}), its $T$-image $T\zeta$ is cartesian too.
		\begin{equation} \label{factorisation of cart}
			\begin{tikzpicture}[textbaseline]
			  \matrix(m)[math35]{ M & A \\ \ps M & \ps M \\};
			  \path[map]  (m-1-1) edge[barred] node[above] {$D$} (m-1-2)
			                      edge[ps] node[left] {$\yon_M$} (m-2-1)
			              (m-1-2) edge[ps] node[right] {$d$} (m-2-2);
			  \path				(m-2-1) edge[eq] (m-2-2);
			  \draw				($(m-1-1)!0.5!(m-2-2)$) node[small] {cart};
			\end{tikzpicture}  = \begin{tikzpicture}[textbaseline]
	  		\matrix(m)[math35]{ M & A \\ M & \ps M \\ \ps M & \ps M \\};
	  		\path[map]	(m-1-1) edge[barred] node[above] {$D$} (m-1-2)
	  								(m-1-2) edge[ps] node[right] {$d$} (m-2-2)
	  								(m-2-1) edge[barred] node[below] {$\yon_{M*}$} (m-2-2)
	  												edge[ps] node[left] {$\yon_M$} (m-3-1);
	  		\path				(m-1-1) edge[eq] (m-2-1)
	  								(m-2-2) edge[eq, ps] (m-3-2)
	  								(m-3-1) edge[eq] (m-3-2);
	  		\path[transform canvas={shift=(m-2-2)}]	(m-1-1) edge[cell] node[right] {$\zeta$} (m-2-1);
	  		\draw				($(m-2-1)!0.5!(m-3-2)$)	node[small] {cart};
	  	\end{tikzpicture}
		\end{equation}
		
		The idea is to find a cell that defines $l \of b$ as a pointwise left Kan extension, and to consider the structure cell $\cell{\bar l}{l \of b}{w \of Tl}$ as a factorisation through this cell. To this end consider the two composites below. Here in the left-hand side the isomorphism $\iota_{A*} \hc Tj_* \iso j_* \hc \iota_{B*}$ is the factorisation of the composite of the opcartesian cells defining $j_*$ and $\iota_{B*}$ through that of those defining $\iota_{A*}$ and $T{j_*}$; remember from the discussion following \eqref{companion pseudofunctor} that both these composites are opcartesian: they both define the companion of $\iota_B \of j = Tj \of \iota_A$. The isomorphism $\iota_{M*} \hc T\yon_{M*} \iso \yon_{M*} \hc \iota_{\ps M*}$ is defined analogously, while the cells $\hat b'$ and $\hat w'$ are factorisations of the unit cells $\hat b$ and $\hat w$ through the opcartesian cells defining the companions $\iota_{B*}$ and $\iota_{\ps M*}$. Notice that both $\hat b'$ and $\hat w'$ are cartesian by \lemref{unit of adjunction as cartesian cell}, as they are the units of the adjunctions $\iota_B \ladj b$ and $\iota_{\ps M} \ladj w$.
		\begin{equation} \label{thm1-identity}
			\begin{tikzpicture}[textbaseline]
				\matrix(m)[math35]{M & A & TA & TB & TB \\ M & A & B & TB & TB \\ M & A & B & B & T\ps M \\ \ps M & \ps M & \ps M & \ps M & \ps M \\};
				\path[map]	(m-1-1) edge[barred] node[above] {$D$} (m-1-2)
										(m-1-2) edge[barred] node[above] {$\iota_{A*}$} (m-1-3)
										(m-1-3) edge[barred] node[above] {$Tj_*$} (m-1-4)
										(m-2-1) edge[barred] node[above] {$D$} (m-2-2)
										(m-2-2) edge[barred] node[above] {$j_*$} (m-2-3)
										(m-2-3) edge[barred] node[above] {$\iota_{B*}$} (m-2-4)
										(m-2-4) edge node[right] {$b$} (m-3-4)
										(m-2-5) edge[ps] node[right] {$Tl$} (m-3-5)
										(m-3-1) edge[barred] node[above] {$D$} (m-3-2)
														edge[ps] node[left] {$\yon_M$} (m-4-1)
										(m-3-2) edge[barred] node[above] {$j_*$} (m-3-3)
														edge[ps] node[right] {$d$} (m-4-2)
										(m-3-3) edge[ps] node[right] {$l$} (m-4-3)
										(m-3-4) edge[ps] node[right] {$l$} (m-4-4)
										(m-3-5) edge[ps] node[right] {$w$} (m-4-5);
				\path				(m-1-1) edge[eq] (m-2-1)
										(m-1-2) edge[eq] (m-2-2)
										(m-1-4) edge[eq] (m-1-5)
														edge[eq] (m-2-4)
										(m-1-5) edge[eq] (m-2-5)
										(m-2-1) edge[eq] (m-3-1)
										(m-2-2) edge[eq] (m-3-2)
										(m-2-3) edge[eq] (m-3-3)
										(m-2-4) edge[eq] (m-2-5)
										(m-3-3) edge[eq] (m-3-4)
										(m-4-1) edge[eq] (m-4-2)
										(m-4-2) edge[eq] (m-4-3)
										(m-4-3) edge[eq] (m-4-4)
										(m-4-4) edge[eq] (m-4-5)
										(m-1-3) edge[cell] node[right] {$\iso$} (m-2-3);
				\path[transform canvas={xshift=1.75em}]	(m-3-2) edge[cell] node[right] {$\eta$} (m-4-2)
										(m-2-3) edge[cell] node[right] {$\hat b'$} (m-3-3);
				\path[transform canvas={shift={(1.75em,-1.625em)}}] (m-2-4) edge[cell] node[right] {$\bar l$} (m-3-4);
				\draw				($(m-3-1)!0.5!(m-4-2)$) node[small] {cart};
			\end{tikzpicture} = \mspace{-12mu} \begin{tikzpicture}[textbaseline]
				\matrix(m)[math35, row sep={3.25em,between origins}, column sep={3.75em,between origins}]{M & A & TA & TB \\ M & TM & TA & TB \\ M & TM & T\ps M & T\ps M \\ M & \ps M & T\ps M & T\ps M \\ \ps M & \ps M & \ps M & \ps M \\};
				\path[map]	(m-1-1) edge[barred] node[above] {$D$} (m-1-2)
										(m-1-2) edge[barred] node[above] {$\iota_{A*}$} (m-1-3)
										(m-1-3) edge[barred] node[above] {$Tj_*$} (m-1-4)
										(m-2-1) edge[barred] node[below] {$\iota_{M*}$} (m-2-2)
										(m-2-2) edge[barred] node[above] {$TD$} (m-2-3)
										(m-2-3) edge[barred] node[above] {$Tj_*$} (m-2-4)
														edge[ps] node[right] {$Td$} (m-3-3)
										(m-2-4) edge[ps] node[right] {$Tl$} (m-3-4)
										(m-3-1) edge[barred] node[above] {$\iota_{M*}$} (m-3-2)
										(m-3-2) edge[barred] node[below] {$T\yon_{M*}$} (m-3-3)
										(m-4-1) edge[barred] node[above] {$\yon_{M*}$} (m-4-2)	
														edge[ps] node[left] {$\yon_M$} (m-5-1)
										(m-4-2) edge[barred] node[below] {$\iota_{\ps M*}$} (m-4-3)
										(m-4-3) edge[ps] node[right] {$w$} (m-5-3)
										(m-4-4) edge[ps] node[right] {$w$} (m-5-4);
				\path				(m-1-1) edge[eq] (m-2-1)
										(m-1-3) edge[eq] (m-2-3)
										(m-1-4) edge[eq] (m-2-4)
										(m-2-1) edge[eq] (m-3-1)
										(m-2-2) edge[eq] (m-3-2)
										(m-3-1) edge[eq] (m-4-1)
										(m-3-3) edge[eq] (m-3-4)
														edge[eq, ps] (m-4-3)
										(m-3-4) edge[eq, ps] (m-4-4)
										(m-4-2) edge[eq, ps] (m-5-2)
										(m-4-3) edge[eq] (m-4-4)
										(m-5-1) edge[eq] (m-5-2)
										(m-5-2) edge[eq] (m-5-3)
										(m-5-3) edge[eq] (m-5-4)
										(m-1-2) edge[cell] node[right] {$\iota_{D*}$} (m-2-2)
										(m-3-2) edge[cell] node[right] {$\iso$} (m-4-2);
				\path[transform canvas={xshift=1.875em}]	(m-2-2) edge[cell] node[right] {$T\zeta$} (m-3-2)
										(m-2-3) edge[cell] node[right] {$T\eta$} (m-3-3)
										(m-4-2) edge[cell, transform canvas={yshift=-2pt}] node[right] {$\hat w'$} (m-5-2);
				\draw				($(m-4-1)!0.5!(m-5-2)$) node[small] {cart};
			\end{tikzpicture}
		\end{equation}
		
		We claim that the composites above coincide. To see this, first notice that we may equivalently show that they coincide after precomposition with the composite of the opcartesian cells defining the companions $\iota_{A*}$ and $Tj_*$, which is itself opcartesian as we remarked above. That they do follows from the equation given in \figref{thm1-proof}. Indeed the top-left and middle-left composites shown there equal respectively the precomposition of the left-hand side and right-hand side above with the opcartesian cells defining $\iota_{A*}$ and $Tj_*$ which, in case of the left-hand side, follows from the definitions of $\iota_{A*} \hc Tj_* \iso j_* \hc \iota_{B*}$, $\hat b'$ and $\bar l$. The identities displayed in \figref{thm1-proof} follow from: (i) one of the triangle identities for $\hat b$ and $\theta_b$; (ii) the naturality of the double transformation $\nat\iota{\id_\K}T$; (iii) the identity \eqref{factorisation of cart} and the companion identities for $\iota_{A*}$ and $\iota_{M*}$ (see \lemref{unit cell factorisation}); (iv) the definition of $\iota_{D*}$ and that of $\iota_{M*} \hc T\yon_{M*} \iso \yon_{M*} \hc \iota_{\ps M*}$, the latter in terms of the cartesian cells that define the companions of $\iota_M$, $T\yon_M$, $\yon_M$ and $\iota_{\ps M}$; (v) the identity $\hat w' = \hat w \hc (1_w \of \textup{cart})$ (which follows from the uniqueness of factorisations through the opcartesian cell defining $\iota_{\ps M*}$).
		\begin{figure}
			\includegraphics{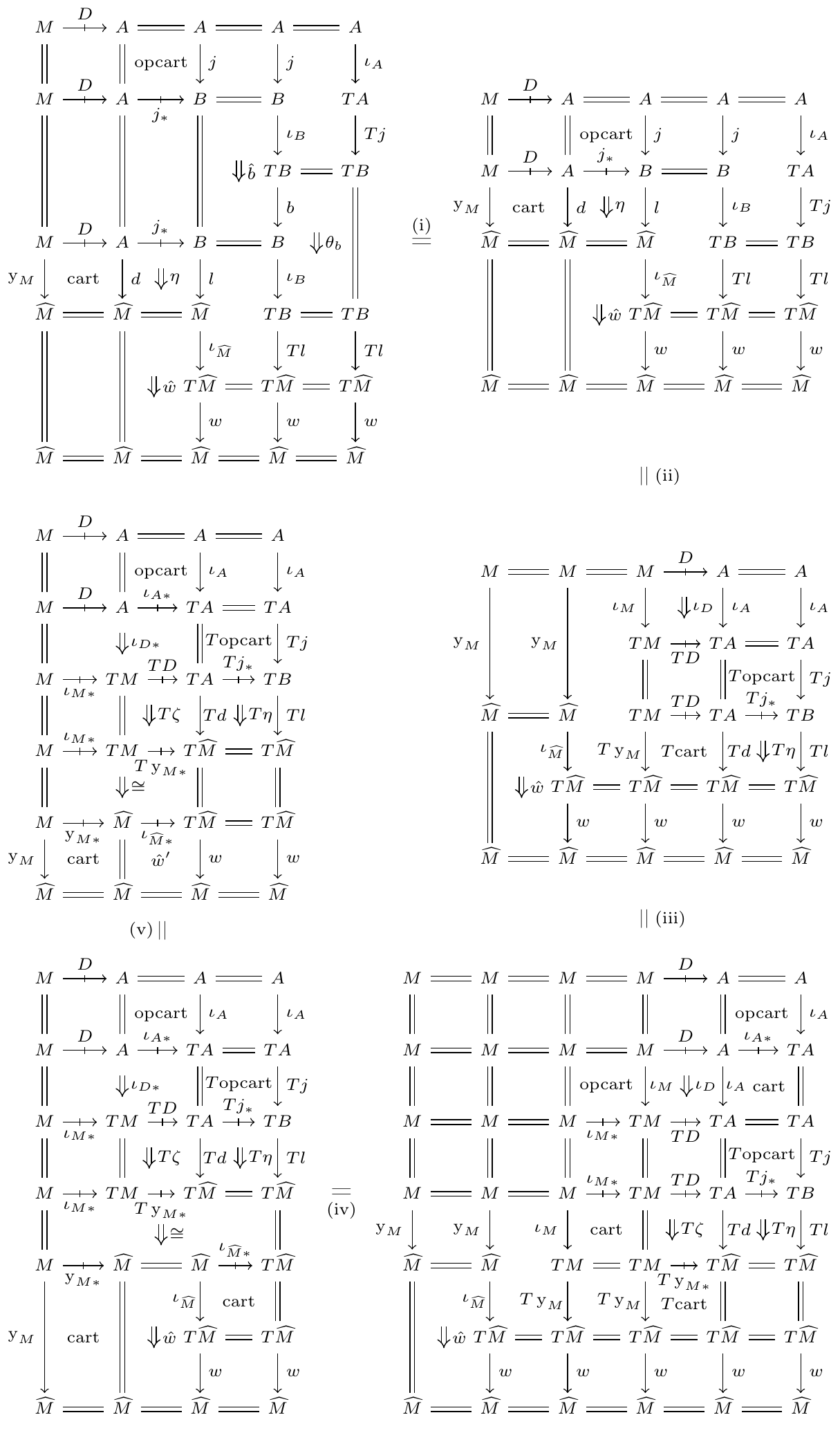}
			\caption{Proof of the identity \eqref{thm1-identity}.}
			\label{thm1-proof}
		\end{figure}		
		
		We denote the composition of the first three columns of the left-hand side of \eqref{thm1-identity} by $\phi$ and that of the bottom three rows of the first two columns of the right-hand side by $\psi$ so that, schematically, the identity above can be redrawn as shown below. We claim that both $\phi$ and $\psi$ are cartesian, so that they define pointwise left Kan extensions by axiom (e) of \defref{yoneda embedding}. To see this first notice that the bottom row of $\phi$ consists of cells defining pointwise left Kan extensions, by axiom (e) and by assumption on $\eta$, so that the full row defines a pointwise left Kan extension by the pasting lemma, and thus is cartesian by \lemref{left Kan extensions along yoneda embeddings}. We conclude that the composites $\phi$ and $\psi$ are build up using cartesian and invertible cells only so that, by \lemref{horizontal composite of cartesian cells} and the pasting lemma, they are cartesian themselves.
		\begin{displaymath}
			\begin{tikzpicture}[baseline=+0.5cm, x=0.75cm, y=0.6cm, font=\scriptsize]
				\draw	(2,0) -- (2,2) -- (0,2) -- (0,0) -- (3,0) -- (3,2) -- (2,2);
				\draw (1,1) node {$\phi$}
							(2.5,1) node {$\bar l$};
			\end{tikzpicture} \mspace{12mu} = \mspace{12mu} \begin{tikzpicture}[baseline=+0.8cm, x=0.75cm, y=0.6cm, font=\scriptsize]
				\draw (0,2) -- (3,2) -- (3,0) -- (0,0) -- (0,3) -- (1.5,3) -- (1.5,0);
				\draw	(0.75,1) node {$\psi$}
							(0.75,2.5) node {$\iota_{D*}$}
							(2.25,1) node {$1_w \of T\eta$};
			\end{tikzpicture}
		\end{displaymath}		
		
		To prove the `whenever'-part, assume that $\iota_D$ satisfies the right Beck-Chevalley condition, that is the top cell $\iota_{D*}$ in the right-hand side above is invertible, and that $1_w \of T\eta$ defines a left Kan extension. Using the pasting lemma we find that the full right-hand side defines a left Kan extension. Since $\phi$ in the left-hand side does too we conclude that $\bar l$ must be invertible, by the uniqueness of Kan extensions.
		
		For the converse assume that $\map jAB$ is full and faithful and that $\bar l$ is invertible. Writing $\eps$ for the opcartesian cell that defines the companion $\hmap{j_*}AB$ of $j$, remember that it is cartesian because $j$ is full and faithful. Using \propref{lax double functors preserve cartesian cells}, it follows that $T\eps$ is cartesian too. The left-hand side above is cartesian because $\bar l$ is invertible so that, when precomposed with $T\eps$, we obtain a composite that is again cartesian. Hence the right-hand side precomposed with $T\eps$, that is \mbox{$(\psi \of \iota_{D*}) \hc (1_w \of T\eta \of T\eps)$}, is cartesian as well. Since $\eta \of \eps$ is invertible by \lemref{pointwise left Kan extension along full and faithful map}, the composite $\psi \hc (1_w \of T\eta \of T\eps)$ in here is cartesian too; we conclude that $\iota_{D*}$ is the factorisation of one cartesian cell through another. Since $\iota_{D*}$ is a horizontal cell it follows that it is invertible. Returning to the identity above, we see that now both the left-hand side as well as the composite $\psi \of \iota_{D*}$ in the right-hand side are cartesian and thus define pointwise left Kan extensions by axiom (e) of \defref{yoneda embedding}. Invoking the pasting lemma for Kan extensions we conclude that $1_w \of T\eta$ defines a pointwise left Kan extension. This completes the proof.
	\end{proof}
	
	\subsection{In the case of a general map \texorpdfstring{$\map dAM$}{d: A → M}}
	The proof of the previous theorem can be modified into one for the following related result, which concerns pointwise left Kan extensions of maps that are of the general form $\map dAM$.
	
	Suppose that the cell $\eta$ on the left below defines $l$ as the pointwise left Kan extension of $d$ along the companion of $j$. Assume that $B$ and $M$ admit pseudo $T$-algebra structures $(b, \bar b, \hat b)$ and $(m, \bar m, \hat m)$, so that $l$ admits a unique structure cell $\cell{\bar l}{l \of b}{m \of Tl}$.
	\begin{displaymath}
			\begin{tikzpicture}[textbaseline]
  			\matrix(m)[math35]{A & B \\ M & M \\};
  			\path[map]  (m-1-1) edge[barred] node[above] {$j_*$} (m-1-2)
														edge node[left] {$d$} (m-2-1)
										(m-1-2) edge node[right] {$l$} (m-2-2);
				\path				(m-2-1) edge[eq] (m-2-2);
				\path[transform canvas={shift={($(m-1-2)!(0,0)!(m-2-2)$)}}] (m-1-1) edge[cell] node[right] {$\eta$} (m-2-1);
			\end{tikzpicture} \qquad\qquad\qquad\qquad\qquad \begin{tikzpicture}[textbaseline]
				\matrix(m)[math35, column sep=1.5em]{TA & TB \\ TM & TM \\ M & M \\};
				\path[map]	(m-1-1) edge[barred] node[above] {$Tj_*$} (m-1-2)
														edge node[left] {$Td$} (m-2-1)
										(m-1-2) edge node[right] {$Tl$} (m-2-2)
										(m-2-1) edge node[left] {$m$} (m-3-1)
										(m-2-2) edge node[right] {$m$} (m-3-2);
				\path				(m-2-1) edge[eq] (m-2-2)
										(m-3-1) edge[eq] (m-3-2);
				\path[transform canvas={shift=(m-2-2)}]	(m-1-1) edge[cell] node[right] {$T\eta$} (m-2-1);
			\end{tikzpicture}
		\end{displaymath}
	\begin{theorem} \label{thm2}
		In the above situation the structure cell $\bar l$ is invertible whenever the cell $\iota_{d^*}$ (where $\hmap{d^*}MA$ is the conjoint of $d$) satisfies the right Beck-Chevalley condition and the composite $1_m \of T\eta$, on the right above, defines $m \of Tl$ as a left Kan extension. The converse holds as soon as $\map jAB$ is full and faithful, and $\iota_{j^*}$ satisfies the right Beck-Chevalley condition.
	\end{theorem}
	If $d$ is of the form $\map dA{\ps M}$ then the hypothesis above is stronger than that of \thmref{thm1}: the right Beck-Chevalley condition for $\iota_{d^*}$ implies that for $\iota_D$, where $D = \ps M(\yon_M, d) \iso \yon_{M*} \hc d^*$ as before; this follows from \lemref{right Beck-Chevalley condition for restrictions} and \propref{right Beck-Chevalley condition for companions}.
	\begin{proof}
		We claim that the following identity holds, where $\hat b'$ and $\hat m'$ are the factorisations of the units $\hat b$ and $\hat m$ through the opcartesian cells defining $\iota_{B*}$ and $\iota_{M*}$ respectively. The proof we leave to the interested reader; it is similar to that of \eqref{thm1-identity}.
		\begin{displaymath}
			\begin{tikzpicture}[textbaseline]
				\matrix(m)[math35]{M & A & TA & TB & TB \\ M & A & B & TB & TB \\ M & A & B & B & TM \\ M & M & M & M & M \\};
				\path[map]	(m-1-1) edge[barred] node[above] {$d^*$} (m-1-2)
										(m-1-2) edge[barred] node[above] {$\iota_{A*}$} (m-1-3)
										(m-1-3) edge[barred] node[above] {$Tj_*$} (m-1-4)
										(m-2-1) edge[barred] node[above] {$d^*$} (m-2-2)
										(m-2-2) edge[barred] node[above] {$j_*$} (m-2-3)
										(m-2-3) edge[barred] node[above] {$\iota_{B*}$} (m-2-4)
										(m-2-4) edge node[right] {$b$} (m-3-4)
										(m-2-5) edge node[right] {$Tl$} (m-3-5)
										(m-3-1) edge[barred] node[above] {$d^*$} (m-3-2)
										(m-3-2) edge[barred] node[above] {$j_*$} (m-3-3)
														edge node[right] {$d$} (m-4-2)
										(m-3-3) edge node[right] {$l$} (m-4-3)
										(m-3-4) edge node[right] {$l$} (m-4-4)
										(m-3-5) edge node[right] {$m$} (m-4-5);
				\path				(m-1-1) edge[eq] (m-2-1)
										(m-1-2) edge[eq] (m-2-2)
										(m-1-4) edge[eq] (m-1-5)
														edge[eq] (m-2-4)
										(m-1-5) edge[eq] (m-2-5)
										(m-2-1) edge[eq] (m-3-1)
										(m-2-2) edge[eq] (m-3-2)
										(m-2-3) edge[eq] (m-3-3)
										(m-2-4) edge[eq] (m-2-5)
										(m-3-1) edge[eq] (m-4-1)
										(m-3-3) edge[eq] (m-3-4)
										(m-4-1) edge[eq] (m-4-2)
										(m-4-2) edge[eq] (m-4-3)
										(m-4-3) edge[eq] (m-4-4)
										(m-4-4) edge[eq] (m-4-5)
										(m-1-3) edge[cell] node[right] {$\iso$} (m-2-3);
				\path[transform canvas={xshift=1.75em}]	(m-3-2) edge[cell] node[right] {$\eta$} (m-4-2)
										(m-2-3) edge[cell] node[right] {$\hat b'$} (m-3-3);
				\path[transform canvas={shift={(1.75em,-1.625em)}}] (m-2-4) edge[cell] node[right] {$\bar l$} (m-3-4);
				\draw				($(m-3-1)!0.5!(m-4-2)$) node[small] {cart};
			\end{tikzpicture} = \begin{tikzpicture}[textbaseline]
				\matrix(m)[math35]{M & A & TA & TB \\ M & TM & TA & TB \\ M & TM & TM & TM \\ M & M & M & M \\};
				\path[map]	(m-1-1) edge[barred] node[above] {$d^*$} (m-1-2)
										(m-1-2) edge[barred] node[above] {$\iota_{A*}$} (m-1-3)
										(m-1-3) edge[barred] node[above] {$Tj_*$} (m-1-4)
										(m-2-1) edge[barred] node[below] {$\iota_{M*}$} (m-2-2)
										(m-2-2) edge[barred] node[below] {$Td^*$} (m-2-3)
										(m-2-3) edge[barred] node[above] {$Tj_*$} (m-2-4)
														edge node[right] {$Td$} (m-3-3)
										(m-2-4) edge node[right] {$Tl$} (m-3-4)
										(m-3-1) edge[barred] node[above] {$\iota_{M*}$} (m-3-2)
										(m-3-2) edge node[right] {$m$} (m-4-2)
										(m-3-3) edge node[right] {$m$} (m-4-3)
										(m-3-4) edge node[right] {$m$} (m-4-4);
				\path				(m-1-1) edge[eq] (m-2-1)
										(m-1-3) edge[eq] (m-2-3)
										(m-1-4) edge[eq] (m-2-4)
										(m-2-1) edge[eq] (m-3-1)
										(m-2-2) edge[eq] (m-3-2)
										(m-3-1) edge[eq] (m-4-1)
										(m-3-2) edge[eq] (m-3-3)
										(m-3-3)	edge[eq] (m-3-4)
										(m-4-1) edge[eq] (m-4-2)
										(m-4-2) edge[eq] (m-4-3)
										(m-4-3) edge[eq] (m-4-4)
										(m-1-2) edge[cell] node[right] {$(\iota_{d^*})_*$} (m-2-2);
				\path[transform canvas={xshift=1.75em}]	(m-2-3) edge[cell] node[right] {$T\eta$} (m-3-3)
										(m-3-1) edge[cell] node[right] {$\hat m'$} (m-4-1);
				\draw				($(m-2-2)!0.5!(m-3-3)$) node[small] {$T$cart};
			\end{tikzpicture}
		\end{displaymath}
		
		Analogous to the proof of \thmref{thm1} we denote the composition of the first three columns of the left-hand side above by $\phi$ and that of the bottom two rows of the first two columns of the right-hand side by $\psi$ so that, schematically, the identity above can be redrawn as shown below. We claim that both $\phi$ and $\psi$ define pointwise left Kan extensions. Indeed, both $\hat b'$ and $\hat m'$ are cartesian by \lemref{cartesian and opcartesian cells in terms of companions and conjoints} so that, by \lemref{conjoints define pointwise left Kan extensions}, each of the columns in the composites $\phi$ and $\psi$ defines a pointwise left Kan extension. Hence the claim follows from the pasting lemma.
		\begin{displaymath}
			\begin{tikzpicture}[baseline=+0.5cm, x=0.75cm, y=0.6cm, font=\scriptsize]
				\draw	(2,0) -- (2,2) -- (0,2) -- (0,0) -- (3,0) -- (3,2) -- (2,2);
				\draw (1,1) node {$\phi$}
							(2.5,1) node {$\bar l$};
			\end{tikzpicture} \mspace{12mu} = \mspace{12mu} \begin{tikzpicture}[baseline=+0.8cm, x=0.75cm, y=0.6cm, font=\scriptsize]
				\draw (0,2) -- (3,2) -- (3,0) -- (0,0) -- (0,3) -- (1.5,3) -- (1.5,0);
				\draw	(0.75,1) node {$\psi$}
							(0.75,2.5) node {$(\iota_{d^*})_*$}
							(2.25,1) node {$1_m \of T\eta$};
			\end{tikzpicture}
		\end{displaymath}
		
		The `whenever'-part of the theorem is now easily proved: if $\iota_{d^*}$ satisfies the right Beck-Chevalley condition, that is $(\iota_{d^*})_*$ is invertible, and $1_m \of T\eta$ defines $m \of Tl$ as a left Kan extension, then the full right-hand side above defines a left Kan extension by the pasting lemma. Since $\phi$ in the left-hand side defines a pointwise left Kan extension too we conclude that $\bar l$ must be invertible, by the uniqueness of Kan extensions.
		
		For the converse assume that $\bar l$ is invertible, that $\map jAB$ is full and faithful, and that $\iota_{j^*}$ satisfies the right Beck-Chevalley condition. It follows from \propref{right Beck-Chevalley condition for colax T-morphisms} that $\iota_{l^*}$ satisfies the right Beck-Chevalley condition so that, by \lemref{right Beck-Chevalley condition for restrictions}, the restriction $l^*(\id, j)$ does too. Now the fact that $j$ is full and faithful means that $d \iso l \of j$ by \lemref{pointwise left Kan extension along full and faithful map} and we conclude that $\iota_{d^*}$ satisfies the right Beck-Chevalley condition as well, by applying the naturality of the double transformation \mbox{$\nat\iota{\id_\K}T$} to the isomorphisms $l^*(\id, j) \iso (l \of j)^* \iso d^*$. Returning to the identity above we see that, now that both $\bar l$ and $(\iota_{d^*})_*$ are invertible, both the left-hand side and the composite $\psi \of (\iota_{d^*})_*$ in the right-hand side define pointwise left Kan extensions so that, by the pasting lemma, the composite $1_m \of T\eta$ defines a pointwise left Kan extension as well. This completes the proof.
	\end{proof}
	
	\subsection{In the case that \texorpdfstring{$A$}{A} admits a pseudo \texorpdfstring{$T$}{T}-algebra structure}
	Finally we consider the `best' situation, in which $A$ admits a pseudo $T$-algebra structure as well. Remember that in this case $\iota_{d^*}$ satisfies the right Beck-Chevalley condition precisely if the unique structure cell $\cell{\bar d}{d \of a}{m \of Td}$, that makes $\map dAM$ into a colax $T$-morphism, is invertible; see \propref{right Beck-Chevalley condition for colax T-morphisms}.
	
	As before, we consider a cell $\eta$ as on the left below and assume that it defines $l$ as the pointwise left Kan extension of $d$ along the companion of $j$. Next we suppose that each of $A$, $B$ and $M$ admit pseudo $T$-algebra structures $(a, \bar a, \hat a)$, $(b, \bar b, \hat b)$ and $(m, \bar m, \hat m)$ respectively, so that the maps $d$, $j$ and $l$ admit unique structure cells $\bar d$, $\bar j$ and $\bar l$, making them into colax $T$-morphisms.
	\begin{displaymath}
		\begin{tikzpicture}[textbaseline]
  			\matrix(m)[math35]{A & B \\ M & M \\};
  			\path[map]  (m-1-1) edge[barred] node[above] {$j_*$} (m-1-2)
														edge node[left] {$d$} (m-2-1)
										(m-1-2) edge node[right] {$l$} (m-2-2);
				\path				(m-2-1) edge[eq] (m-2-2);
				\path[transform canvas={shift={($(m-1-2)!(0,0)!(m-2-2)$)}}] (m-1-1) edge[cell] node[right] {$\eta$} (m-2-1);
			\end{tikzpicture} \qquad\qquad\qquad\qquad\qquad \begin{tikzpicture}[textbaseline]
				\matrix(m)[math35, column sep=1.5em]{TA & TB \\ TM & TM \\ M & M \\};
				\path[map]	(m-1-1) edge[barred] node[above] {$Tj_*$} (m-1-2)
														edge node[left] {$Td$} (m-2-1)
										(m-1-2) edge node[right] {$Tl$} (m-2-2)
										(m-2-1) edge node[left] {$m$} (m-3-1)
										(m-2-2) edge node[right] {$m$} (m-3-2);
				\path				(m-2-1) edge[eq] (m-2-2)
										(m-3-1) edge[eq] (m-3-2);
				\path[transform canvas={shift=(m-2-2)}]	(m-1-1) edge[cell] node[right] {$T\eta$} (m-2-1);
			\end{tikzpicture}
		\end{displaymath}
	\begin{theorem}\label{thm3}
		In the above situation assume that the structure cell $\bar d$ is invertible. The structure cell $\bar l$ is invertible if and only if the composite $1_m \of T\eta$ on the right above defines $m \of Tl$ as a left Kan extension.
	\end{theorem}
	\begin{proof}
		We consider the two composites in the identity below where, in the left-hand side, the cell $\bar{j_*}$ is taken to be the factorisation of $\bar j$ that is defined by the identity on the left of \eqref{companion algebra structure} below.
		\begin{equation} \label{thm3-identity}
			\begin{tikzpicture}[textbaseline]
  			\matrix(m)[math35, column sep={3.75em,between origins}]{TA & TB & TB \\ A & B & TM \\ M & M & M \\};
  			\path[map]	(m-1-1) edge[barred] node[above] {$Tj_*$} (m-1-2)
  													edge node[left] {$a$} (m-2-1)
  									(m-1-2) edge node[right] {$b$} (m-2-2)
  									(m-1-3) edge node[right] {$Tl$} (m-2-3)
  									(m-2-1) edge[barred] node[below] {$j_*$} (m-2-2)
  													edge node[left] {$d$} (m-3-1)
  									(m-2-2) edge node[right] {$l$} (m-3-2)
  									(m-2-3) edge node[right] {$m$} (m-3-3);
  			\path				(m-1-2) edge[eq] (m-1-3)
  									(m-3-1) edge[eq] (m-3-2)
  									(m-3-2) edge[eq] (m-3-3);
  			\path[transform canvas={xshift=1.875em}]	(m-1-1) edge[cell] node[right] {$\bar{j_*}$} (m-2-1)
  									(m-2-1) edge[cell, transform canvas={yshift=-3pt}] node[right] {$\eta$} (m-3-1);
  			\path[transform canvas={shift={(1.875em, -1.625em)}}]	(m-1-2) edge[cell] node[right] {$\bar l$} (m-2-2);
  		\end{tikzpicture} = \begin{tikzpicture}[textbaseline]
  			\matrix(m)[math35, column sep={3.75em,between origins}]{TA & TA & TB \\ A & TM & TM \\ M & M & M \\};
  			\path[map]	(m-1-1) edge node[left] {$a$} (m-2-1)
  									(m-1-2) edge[barred] node[above] {$Tj_*$} (m-1-3)
  													edge node[left] {$Td$} (m-2-2)
  									(m-1-3) edge node[right] {$Tl$} (m-2-3)
  									(m-2-1) edge node[left] {$d$} (m-3-1)
  									(m-2-2) edge node[left] {$m$} (m-3-2)
  									(m-2-3) edge node[right] {$m$} (m-3-3);
  			\path				(m-1-1) edge[eq] (m-1-2)
  									(m-2-2) edge[eq] (m-2-3)
  									(m-3-1) edge[eq] (m-3-2)
  									(m-3-2) edge[eq] (m-3-3);
  			\path[transform canvas={xshift=1.875em}]	(m-1-2) edge[cell] node[right] {$T\eta$} (m-2-2);
  			\path[transform canvas={shift={(1.375em, -1.625em)}}]	(m-1-1) edge[cell] node[right] {$\bar d$} (m-2-1);
  		\end{tikzpicture}
  	\end{equation}
  	To see that the identity holds, note that its two sides coincide after precomposing them with the $T$-image of the opcartesian cell defining the companion of $j$, as is shown in \figref{thm3-proof}, so that the identity follows by uniqueness of factorisations through opcartesian cells. Indeed, both the left-hand and right-hand side of \figref{thm3-proof} coincide with those above after this precomposition, as follows from the companion identity for $j_*$ and the definition of $\bar{j_*}$, and by writing out the structure cells $\bar l$, $\bar j$ and $\bar d$ (as in \eqref{colax algebra structure}). The identities of \figref{thm3-proof} follow from one of the triangle identities for $\hat b$ and $\theta_b$, that define the adjunction $\iota_B \ladj b$, and the naturality of $\iota$ respectively.
  	\begin{equation} \label{companion algebra structure}
  		\bar j  = \mspace{-9mu}\begin{tikzpicture}[textbaseline]
  			\matrix(m)[math35]{TA & TA\\ TA & TB\\ A & B\\ B & B\\};
  			\path[map]	(m-1-2) edge node[right] {$Tj$} (m-2-2)
  									(m-2-1) edge[barred] node[below] {$Tj_*$} (m-2-2)
  													edge node[left] {$a$} (m-3-1)
  									(m-2-2) edge node[right] {$b$} (m-3-2)
  									(m-3-1) edge[barred] node[below] {$j_*$} (m-3-2)
  													edge node[left] {$j$} (m-4-1);
  			\path				(m-1-1) edge[eq] (m-1-2)
  													edge[eq] (m-2-1)
  									(m-3-2) edge[eq] (m-4-2)
  									(m-4-1) edge[eq] (m-4-2);
  			\path[transform canvas={xshift=1.75em, yshift=-3pt}]	(m-2-1) edge[cell] node[right] {$\bar{j_*}$} (m-3-1);
  			\path[transform canvas={shift={(1.75em, -1.625em)}}]	(m-1-1) node[small] {$T$opcart}
  									(m-3-1) node[small] {cart};
  		\end{tikzpicture} \qquad\qquad (\bar{j_*})^* = \begin{tikzpicture}[textbaseline]
	  		\matrix(m)[math35]{A & TA & TB & TB \\ A & A & B & TB \\};
	  		\path[map]	(m-1-1)	edge[barred] node[above] {$\iota_{A*}$} (m-1-2)
	  								(m-1-2) edge[barred] node[above] {$Tj_*$} (m-1-3)
	  												edge node[right] {$a$} (m-2-2)
	  								(m-1-3) edge node[right] {$b$} (m-2-3)
	  								(m-2-2) edge[barred] node[below] {$j_*$} (m-2-3)
	  								(m-2-3) edge[barred] node[below] {$\iota_{B*}$} (m-2-4);
	  		\path				(m-1-1) edge[eq] (m-2-1)
	  								(m-1-3)	edge[eq] (m-1-4)
	  								(m-1-4) edge[eq] (m-2-4)
	  								(m-2-1) edge[eq] (m-2-2);
	  		\path[transform canvas={xshift=1.75em}]	(m-1-1) edge[cell] node[right] {$\hat a'$} (m-2-1)
	  								(m-1-2) edge[cell] node[right] {$\bar{j_*}$} (m-2-2)
	  								(m-1-3) edge[cell] node[right] {$\theta_b'$} (m-2-3);
	  	\end{tikzpicture}
  	\end{equation}
  	  		
  	Next we claim that $\bar{j_*}$ satisfies the \emph{left Beck-Chevalley condition}, that is the composite $(\bar{j_*})^*$ on the right above is invertible. Here $\hat a'$ denotes the factorisation of $\hat a$ through the opcartesian cell defining $\iota_{A*}$, while $\theta_b'$ denotes the factorisation of $\theta_b$ through the cartesian cell defining $\iota_{B*}$. Recall that $\hat a'$ is cartesian  by \lemref{unit of adjunction as cartesian cell} while $\theta_b'$ is opcartesian by its horizontal and vertical dual; now compare $(\bar{j_*})^*$ with the composite $\iota_{J*}$ of \eqref{right Beck-Chevalley composite}, that was considered in the definition of the right Beck-Chevalley condition on $\iota_J$. To prove the claim simply precompose $(\bar{j_*})^*$ with the composite of the opcartesian cells defining $\iota_{A*}$ and $Tj_*$, and postcompose it with the composite of the cartesian cells defining $j_*$ and $\iota_{B*}$; remember that these composites are (op\ndash )cartesian by the discussion following \eqref{companion pseudofunctor}. The resulting composite is simply the identity $\iota_B \of j = Tj \of \iota_A$, as is easily checked by using the triangle identities for $(\hat a, \theta_a)$ and $(\hat b, \theta_b)$. Using the uniqueness of factorisations through (op-)cartesian cells we conclude that $(\bar{j_*})^*$ coincides with the canonical isomorphism $\iota_{A*} \hc Tj_* \iso j_* \hc \iota_{B*}$.
  	
  	As a consequence of the left Beck-Chevalley condition for $(\bar{j_*})^*$ the composite $\eta \of \bar{j_*}$, in the left-hand side of \eqref{thm3-identity}, defines $l \of b$ as the pointwise left Kan extension of $d \of a$ along $Tj_*$; see (the horizontal dual of) Corollary 4.6 of \cite{Koudenburg14a}. The proof is now easy to complete: if $\bar d$ is invertible and $1_m \of T\eta$ defines a left Kan extension then both $\eta \of \bar{j_*}$ in the left-hand side as well as the full right-hand side of \eqref{thm3-identity} define left Kan extensions, so that $\bar l$ must be invertible by the uniqueness of Kan extensions. Conversely, assume that $\bar d$ and $\bar l$ are invertible. Then the left-hand side of \eqref{thm3-identity} defines $m \of Tl$ as a pointwise left Kan extension so that, by precomposing both sides with the inverse of $\bar d$, we see that $1_m \of T\eta$ defines $m \of Tl$ as a pointwise left Kan extension too.
  	\begin{figure}[t]
			\includegraphics{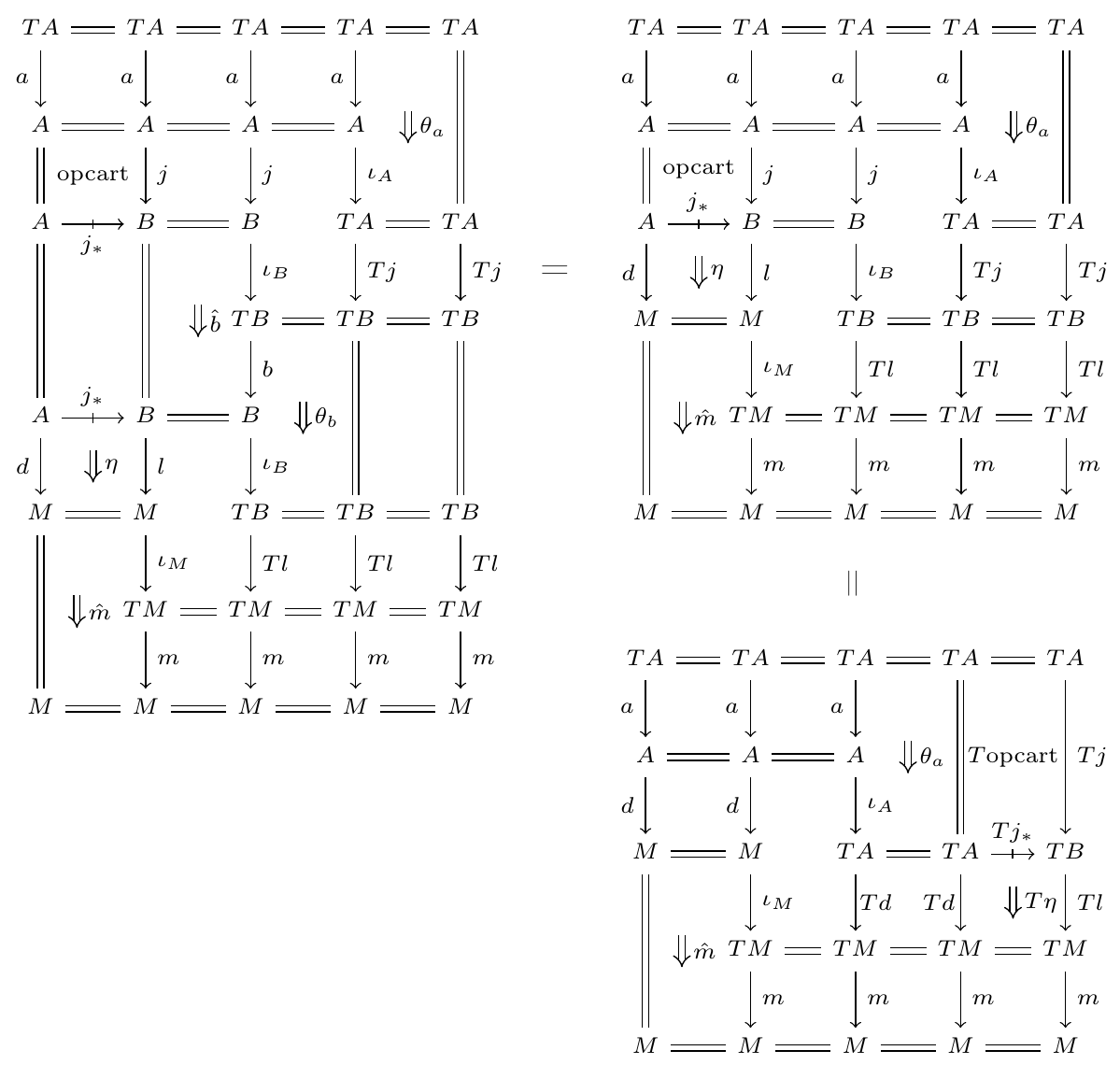}
			\caption{Proof of the identity \eqref{thm3-identity}.}
			\label{thm3-proof}
		\end{figure}
	\end{proof}
	\appendix
	\section{Algebra structures on \texorpdfstring{$\ps A$}{Â}}
	In this appendix we consider algebra structures on presheaf objects $\ps A$, that are defined by a normal lax double monad that restricts to a colax-idempotent $2$-monad. This generalises the classic situation of categories of presheaves $\ps A = \brks{\op A, \Set}$ admitting finite products that are defined pointwise (see \exref{pointwise finite products of presheaves}).
	
	Let $T = (T, \mu, \iota)$ be a normal lax double monad on an equipment $\K$ and assume that it restricts to a colax-idempotent $2$-monad on $V(\K)$; let $\map{\yon_A}A{\ps A}$ be a yoneda embedding in $\K$. By axiom (c) of \defref{yoneda embedding} there exists a cartesian cell
	\begin{displaymath}
		\begin{tikzpicture}
			\matrix(m)[math35]{A & \ps A & T\ps A \\ \ps A & & \ps A, \\};
			\path[map]	(m-1-1) edge[barred] node[above] {$\yon_{A*}$} (m-1-2)
													edge[ps] node[left] {$\yon_A$} (m-2-1)
									(m-1-2) edge[barred] node[above] {$\iota_{\ps A*}$} (m-1-3)
									(m-1-3) edge[ps] node[right] {$w$} (m-2-3);
			\path				(m-2-1) edge[eq] (m-2-3);
			\draw[transform canvas={yshift=-1.625em}]	(m-1-2) node[small] {cart};
		\end{tikzpicture}
	\end{displaymath}
	whose vertical target $\map w{T\ps A}{\ps A}$ is unique up to isomorphism by its axiom (e).
	\begin{proposition} \label{colax-idempotent algebra structure on presheaves object}
		In the above situation $w$ can be extended to a lax $T$-algebra structure $(w, \bar w, \hat w)$ on $\ps A$ whose associator $\cell{\bar w}{w \of Tw}{w \of \mu_{\ps A}}$ is invertible.  Its unitor $\cell{\hat w}{\id_{\ps A}}{w \of \iota_{\ps A}}$ is invertible precisely if $\map{\iota_{\ps A}}{\ps A}{T\ps A}$ is full and faithful.
	\end{proposition}
	
	\begin{example} \label{pointwise finite products of presheaves}
		If $T$ is the `free category with strictly associative finite pro\-ducts'\ndash mo\-nad on $\Prof$, that we have considered throughout, and $\map{\yon_A}A{\ps A = \brks{\op A, \Set}}$ is the usual yoneda embedding, then the action $\map w{T\ps A}{\ps A}$ maps any sequence $(p_1, \dotsc, p_n)$ of presheaves on $A$ to the presheaf given by
		\begin{displaymath}
			w(p_1, \dotsc, p_n)(x) = \prod_{j = 1}^n \ps A(\yon_A x, p_j) \iso \prod_{j = 1}^n p_j x;
		\end{displaymath}
		that is $w$ maps $(p_1, \dotsc, p_n)$ to their pointwise product, as expected.
	\end{example}
	
	\begin{proof}[Proof of \propref{colax-idempotent algebra structure on presheaves object}]
		To define the unitor $\hat w$ notice that the cartesian cell defining $\yon_{A*}$ defines $\id_{\ps A}$ as a pointwise left Kan extension, by axiom (a) of \defref{yoneda embedding}, so that the cartesian cell defining $w$ factors through it as a cell $\zeta$, as on the left below. We define $\hat w$ as the composite on the right.
		\begin{equation} \label{definition unitor}
				\begin{tikzpicture}[textbaseline]
				\matrix(m)[math35]{A & \ps A & T\ps A \\ \ps A & & \ps A \\};
				\path[map]	(m-1-1) edge[barred] node[above] {$\yon_{A*}$} (m-1-2)
														edge[ps] node[left] {$\yon_A$} (m-2-1)
										(m-1-2) edge[barred] node[above] {$\iota_{\ps A*}$} (m-1-3)
										(m-1-3) edge[ps] node[right] {$w$} (m-2-3);
				\path				(m-2-1) edge[eq] (m-2-3);
				\draw[transform canvas={yshift=-1.625em}]	(m-1-2) node[small] {cart};
			\end{tikzpicture} = \begin{tikzpicture}[textbaseline]
				\matrix(m)[math35]{A & \ps A & T\ps A \\ \ps A & \ps A & \ps A \\};
				\path[map]	(m-1-1) edge[barred] node[above] {$\yon_{A*}$} (m-1-2)
														edge[ps] node[left] {$\yon_A$} (m-2-1)
										(m-1-2) edge[barred] node[above] {$\iota_{\ps A*}$} (m-1-3)
										(m-1-3) edge[ps] node[right] {$w$} (m-2-3);
				\path				(m-1-2) edge[ps, eq] (m-2-2)
										(m-2-1) edge[eq] (m-2-2)
										(m-2-2) edge[eq] (m-2-3);
				\path[transform canvas={xshift=1.75em}]	(m-1-2) edge[cell] node[right] {$\zeta$} (m-2-2);
				\draw[transform canvas={shift={(1.75em, -1.625em)}}]	(m-1-1) node[small] {cart};
			\end{tikzpicture} \quad\qquad \hat w = \begin{tikzpicture}[textbaseline]
				\matrix(m)[math35]{\ps A & \ps A \\ \ps A & T\ps A \\ \ps A & \ps A \\};
				\path[map]	(m-1-2) edge node[right] {$\iota_{\ps A}$} (m-2-2)
										(m-2-1) edge[barred] node[below] {$\iota_{\ps A*}$} (m-2-2)
										(m-2-2) edge[ps] node[right] {$w$} (m-3-2);
				\path				(m-1-1) edge[eq] (m-1-2)
														edge[eq, ps] (m-2-1)
										(m-2-1) edge[eq, ps] (m-3-1)
										(m-3-1) edge[eq] (m-3-2);
				\path[transform canvas={xshift=1.75em, yshift=-2pt}]	(m-2-1) edge[cell] node[right] {$\zeta$} (m-3-1);
				\path[transform canvas={shift={(1.75em, -1.625em)}}]	(m-1-1) node[small] {opcart};
			\end{tikzpicture}
		\end{equation}
		
		Before defining the associator $\bar w$ we describe some useful properties of the cell $\zeta$. First notice that $\zeta$ defines $w$ as the pointwise left Kan extension of $\id_{\ps A}$ along $\iota_{\ps A*}$, by applying the pasting lemma to the identity on the left above. Likewise, by applying the pasting lemma for cartesian cells to the same identity, we see that the composite $\id_{\yon_{A*}}\!\! \hc \zeta$ is cartesian, and so is its $T$-image by \propref{lax double functors preserve cartesian cells}. Since the compositor $\cell{T_\hc}{T\yon_{A*} \hc T\iota_{\ps A*}}{T(\yon_{A*} \hc \iota_{\ps A*})}$ is invertible by \lemref{compositor on companions}, we conclude that the composite $T(\id_{\yon_{A*}}\!\! \hc \zeta) \of T_\hc = \id_{T\yon_{A*}}\!\! \hc T\zeta$ is cartesian as well (the identity here is the naturality of $T_\hc$). In the same way the composite $\id_{T^2\yon_{A*}}\!\! \hc T^2\zeta$ is cartesian.
		
		The associator $\bar w$ is defined as the factorisation below, as will be explained. The isomorphism in the composite on the left-hand side is induced by the identity $\iota_{T\ps A} \of \iota_{\ps A} = T\iota_{\ps A} \of \iota_{\ps A}$, which is a consequence of the naturality of $\iota$, while the cell $\eps'$ denotes the factorisation of the unit cell on $\id_{T\ps A} = \mu_{\ps A} \of \iota_{T\ps A}$, which we denote by $\eps$, through the opcartesian cell defining $\iota_{T\ps A*}$. Since $T$ restricts to a colax-idempotent $2$-monad on $\K$, the cell $\eps$ forms the unit of the adjunction $\iota_{T\ps A} \ladj \mu_{\ps A}$; see Theorem~6.2 of \cite{Kelly-Lack97} for the dual result for lax-idempotent $2$-monads. Using \lemref{unit of adjunction as cartesian cell} we conclude that the factorisation $\eps'$ of $\eps$ is cartesian.
		
		To show that the composite on the left-hand side indeed factors through that of the first two columns on the right-hand side, resulting in the cell $\bar w$, we will show that both composites define pointwise left Kan extensions. Notice that this implies that $\bar w$ is invertible, as asserted. That the left-hand side defines a pointwise left Kan extension follows easily by applying the pasting lemma to its bottom two rows, both of whose columns define pointwise left Kan extensions: we have already seen that $\zeta$ does, while $1_w \of \eps'$ does by \lemref{conjoints define pointwise left Kan extensions}, as $\eps'$ is cartesian.
		\begin{displaymath}
			\begin{tikzpicture}[textbaseline]
				\matrix(m)[math35, row sep={3.25em,between origins}, column sep={3.75em,between origins}]
				{ \ps A & T\ps A & T^2\ps A \\
					\ps A & T\ps A & T^2\ps A \\
					\ps A & T\ps A & T\ps A \\
					\ps A & \ps A & \ps A \\ };
				\path[map]	(m-1-1) edge[barred] node[above] {$\iota_{\ps A*}$} (m-1-2)
										(m-1-2) edge[barred] node[above] {$T\iota_{\ps A*}$} (m-1-3)
										(m-2-1) edge[barred] node[below] {$\iota_{\ps A*}$} (m-2-2)
										(m-2-2) edge[barred] node[below] {$\iota_{T\ps A*}$} (m-2-3)
										(m-2-3) edge node[right] {$\mu_{\ps A}$} (m-3-3)
										(m-3-1) edge[barred] node[above] {$\iota_{\ps A*}$} (m-3-2)
										(m-3-2) edge[ps] node[right] {$w$} (m-4-2)
										(m-3-3) edge[ps] node[right] {$w$} (m-4-3);
				\path				(m-1-1) edge[ps, eq] (m-2-1)
										(m-1-3) edge[eq] (m-2-3)
										(m-2-1) edge[ps, eq] (m-3-1)
										(m-2-2) edge[eq] (m-3-2)
										(m-3-1) edge[ps, eq] (m-4-1)
										(m-3-2) edge[eq] (m-3-3)
										(m-4-1) edge[eq] (m-4-2)
										(m-4-2) edge[eq] (m-4-3)
										(m-1-2) edge[cell] node[right] {$\iso$} (m-2-2);
				\path[transform canvas={xshift=1.825em}]	(m-2-2) edge[cell, transform canvas={yshift=-2pt}] node[right] {$\eps'$} (m-3-2)
										(m-3-1) edge[cell] node[right] {$\zeta$} (m-4-1);
			\end{tikzpicture} = \begin{tikzpicture}[textbaseline]
				\matrix(m)[math35, row sep={3.25em,between origins}, column sep={3.75em,between origins}]
				{	\ps A & T\ps A & T^2\ps A & T^2\ps A \\
					\ps A & T\ps A & T\ps A & T\ps A \\
					\ps A & \ps A & \ps A & \ps A \\ };
				\path[map]	(m-1-1) edge[barred] node[above] {$\iota_{A*}$} (m-1-2)
										(m-1-2) edge[barred] node[above] {$T\iota_{A*}$} (m-1-3)
										(m-1-3) edge node[right] {$Tw$} (m-2-3)
										(m-1-4) edge node[right] {$\mu_{\ps A}$} (m-2-4)
										(m-2-1) edge[barred] node[above] {$\iota_{\ps A*}$} (m-2-2)
										(m-2-2) edge[ps] node[right] {$w$} (m-3-2)
										(m-2-3) edge[ps] node[right] {$w$} (m-3-3)
										(m-2-4) edge[ps] node[right] {$w$} (m-3-4);
				\path				(m-1-1) edge[ps, eq] (m-2-1)
										(m-1-2) edge[eq] (m-2-2)
										(m-1-3) edge[eq] (m-1-4)
										(m-2-1) edge[ps, eq] (m-3-1)
										(m-2-2) edge[eq] (m-2-3)
										(m-3-1) edge[eq] (m-3-2)
										(m-3-2) edge[eq] (m-3-3)
										(m-3-3) edge[eq] (m-3-4);
				\path[transform canvas={xshift=1.825em}]	(m-1-2) edge[cell] node[right] {$T\zeta$} (m-2-2)
										(m-2-1) edge[cell] node[right] {$\zeta$} (m-3-1);
				\path[transform canvas={shift={(1.825em, -1.625em)}}]	(m-1-3) edge[cell] node[right] {$\bar w$} (m-2-3);
			\end{tikzpicture}
		\end{displaymath}
		To see that the composite of the first two columns on the right-hand side defines a pointwise left Kan extension notice that, by using the pasting lemma, we may equivalently show that the composite on the left below does so. Of the latter we already know that the bottom row defines a pointwise left Kan extension so that, by \lemref{restriction of pointwise left Kan extensions}, it suffices to show that its top row is cartesian. That it is follows from the fact that $\id_{T\yon_{A*}} \hc T\zeta$ is cartesian, as we have seen, together with the isomorphism $\yon_{A*} \hc \iota_{\ps A*} \iso \iota_{A*} \hc T\yon_{A*}$, which follows from the naturality of $\iota$.
		
		This completes the construction of the invertible associator $\bar w$; what remains is showing that, together with $\hat w$, it satisfies the associativity and unit axioms, see e.g.\ Definition 2.15 of \cite{Koudenburg14b}. In doing so it is useful to consider the identity obtained by precomposing both sides above with the opcartesian cells that define $\iota_{\ps A*}$ and $T\iota_{\ps A*}$; the resulting identity is drawn schematically on the right below, where the identity cell on the left-hand side is that of $\iota_{T\ps A} \of \iota_{\ps A} = T\iota_{\ps A} \of \iota_{\ps A}$.
		\begin{equation} \label{associator definition}
			\begin{tikzpicture}[textbaseline]
				\matrix(m)[math35, row sep={3.25em,between origins}, column sep={3.75em,between origins}]
				{	A & \ps A & T\ps A & T^2\ps A \\
					A & \ps A & T\ps A & T\ps A \\
					\ps A & \ps A & \ps A & \ps A \\ };
				\path[map]	(m-1-1) edge[barred] node[above] {$\yon_{A*}$} (m-1-2)
										(m-1-2) edge[barred] node[above] {$\iota_{\ps A*}$} (m-1-3)
										(m-1-3) edge[barred] node[above] {$T\iota_{\ps A*}$} (m-1-4)
										(m-1-4) edge node[right] {$Tw$} (m-2-4)
										(m-2-1) edge[barred] node[above] {$\yon_{A*}$} (m-2-2)
														edge[ps] node[left] {$\yon_A$} (m-3-1)
										(m-2-2) edge[barred] node[above] {$\iota_{\ps A*}$} (m-2-3)
										(m-2-3) edge[ps] node[right] {$w$} (m-3-3)
										(m-2-4) edge[ps] node[right] {$w$} (m-3-4);
				\path				(m-1-1) edge[eq] (m-2-1)
										(m-1-2) edge[ps, eq] (m-2-2)
										(m-1-3) edge[eq] (m-2-3)
										(m-2-2) edge[ps, eq] (m-3-2)
										(m-2-3) edge[eq] (m-2-4)
										(m-3-1) edge[eq] (m-3-2)
										(m-3-2) edge[eq] (m-3-3)
										(m-3-3) edge[eq] (m-3-4);
				\path[transform canvas={xshift=1.825em}]	(m-1-3) edge[cell] node[right] {$T\zeta$} (m-2-3)
										(m-2-2) edge[cell] node[right] {$\zeta$} (m-3-2);
				\draw[transform canvas={shift={(1.825em, -1.625em)}}]	(m-2-1) node[small] {cart};
			\end{tikzpicture} \qquad\qquad \begin{tikzpicture}[baseline=+1.1cm, x=0.6cm, y=0.6cm, font=\scriptsize]
				\draw	(1,3) -- (1,0) -- (0,0) -- (0,4) -- (1,4) -- (1,3) -- (2,3) -- (2,1) -- (1,1)
							(2,3) -- (2,4) -- (3,4) -- (3,2) -- (2,2);
				\draw (0.5,2) node {$\hat w$}
							(1.5,2) node {$\eps$};
			\end{tikzpicture} \mspace{12mu} = \mspace{12mu} \begin{tikzpicture}[baseline=+1.1cm, x=0.6cm, y=0.6cm, font=\scriptsize]
				\draw (2,1) -- (1,1) -- (1,3) -- (2,3) -- (2,0) -- (3,0) -- (3,2) -- (2,2)
							(1,1) -- (1,0) -- (0,0) -- (0,4) -- (1,4) -- (1,3);
				\draw	(0.5,2) node {$\hat w$}
							(1.5,2) node {$T\hat w$}
							(2.5,1) node {$\bar w$};
			\end{tikzpicture}
		\end{equation}
		
		The three axioms that we need to check are drawn schematically below where, as always, empty cells depict identity cells. For the first two, which are the unit axioms, notice that we may equivalently prove that they hold after composing each of their sides on the left with $\zeta$, which follows from the uniqueness of factorisations through cells defining a Kan extension. The resulting identities in turn hold precisely if they hold after precomposing each side with the opcartesian cell that defines $\iota_{\ps A*}$, which follows from the uniqueness of factorisations through opcartesian cells. Remembering that the unitor $\hat w$ is defined as the composite of $\zeta$ with this opcartesian cell, see \eqref{definition unitor}, we conclude that the unit axioms hold if and only if they hold after composing their sides on the left with $\hat w$. That the first one holds after this composition is a direct consequence of the identity on the right above.
		\begin{equation} \label{lax algebra axioms}
			\begin{tikzpicture}[baseline=+0.8cm, x=0.6cm, y=0.6cm, font=\scriptsize]
				\draw	(1,1) -- (0,1) -- (0,3) -- (1,3) -- (1,0) -- (2,0) -- (2,3) -- (3,3) -- (3,1) -- (2,1)
							(1,2) -- (2,2);
				\draw	(0.5,2) node {$T\hat w$}
							(1.5,1) node {$\bar w$};
			\end{tikzpicture} \mspace{4mu} = 1_w \mspace{18mu} \begin{tikzpicture}[baseline=+0.8cm, x=0.6cm, y=0.6cm, font=\scriptsize]
				\draw	(1,2) -- (0,2) -- (0,0) -- (1,0) -- (1,3) -- (2,3) -- (2,0) -- (3,0) -- (3,2) -- (2,2)
							(1,1) -- (2,1);
				\draw	(0.5,1) node {$\hat w$}
							(2.5,1) node {$\bar w$};
			\end{tikzpicture} \mspace{4mu} = \mspace{4mu} \begin{tikzpicture}[baseline=+0.8cm, x=0.6cm, y=0.6cm, font=\scriptsize]
				\draw (1,3) -- (1,0) -- (0,0) -- (0,3) -- (2,3) -- (2,1) -- (1,1);
				\draw (1.5, 2) node {$\eps$};
			\end{tikzpicture} \mspace{30mu} \begin{tikzpicture}[baseline=+0.8cm, x=0.6cm, y=0.6cm, font=\scriptsize]
				\draw (1,1) -- (0,1) -- (0,3) -- (1,3) -- (1,0) -- (2,0) -- (2,3) -- (3,3) -- (3,1) -- (2,1)
							(1,2) -- (2,2);
				\draw (0.5,2) node {$T\bar w$}
							(1.5,1) node {$\bar w$};
			\end{tikzpicture} \mspace{4mu} = \mspace{4mu} \begin{tikzpicture}[baseline=+0.8cm, x=0.6cm, y=0.6cm, font=\scriptsize]
				\draw	(1,2) -- (0,2) -- (0,0) -- (1,0) -- (1,3) -- (2,3) -- (2,0) -- (3,0) -- (3,2) -- (2,2)
							(1,1) -- (2,1);
				\draw	(0.5,1) node {$\bar w$}
							(2.5,1) node {$\bar w$};
			\end{tikzpicture}
		\end{equation}
		
		That the second unit axiom holds after precomposition with $\hat w$ is shown by the following equation. Here the first identity follows from the naturality of $\iota$ with respect to the top unitor $\bar w$, the empty cell in the second composite depicts the identity cell $T\iota_{\ps A} \of \iota_{\ps A} = \iota_{T\ps A} \of \iota_{\ps A}$, and the second identity follows from the identity on the right of \eqref{associator definition}. 
		\begin{displaymath}
			\begin{tikzpicture}[baseline=+1.1cm, x=0.6cm, y=0.6cm, font=\scriptsize]
				\draw	(2,2) -- (3,2) -- (3,0) -- (2,0) -- (2,3) -- (1,3) -- (1,0) -- (0,0) -- (0,4) -- (1,4) -- (1,3)
							(0,2) -- (1,2)
							(1,1) -- (2,1);
				\draw	(0.5,1) node {$\hat w$}
							(0.5,3) node {$\hat w$}
							(2.5,1) node {$\bar w$};
			\end{tikzpicture} \mspace{9mu} = \mspace{9mu} \begin{tikzpicture}[baseline=+1.1cm, x=0.6cm, y=0.6cm, font=\scriptsize]
				\draw	(1,3) -- (1,0) -- (0,0) -- (0,4) -- (1,4) -- (1,3) -- (2,3) -- (2,4) -- (3,4) -- (3,0) -- (2,0) -- (2,3)
							(1,1) -- (2,1)
							(2,2) -- (3,2);
				\draw (0.5,2) node {$\hat w$}
							(1.5,2) node {$T\hat w$}
							(2.5,1) node {$\bar w$};
			\end{tikzpicture} \mspace{9mu} = \mspace{9mu} \begin{tikzpicture}[baseline=+1.1cm, x=0.6cm, y=0.6cm, font=\scriptsize]
				\draw (1,1) -- (2,1) -- (2,3) -- (1,3) -- (1,0) -- (0,0) -- (0,4) -- (1,4) -- (1,3);
				\draw (0.5,2) node {$\hat w$}
							(1.5,2) node {$\eps$};
			\end{tikzpicture} \mspace{9mu} = \mspace{9mu} \begin{tikzpicture}[baseline=+1.1cm, x=0.6cm, y=0.6cm, font=\scriptsize]
				\draw (1,2) -- (0,2) -- (0,4) -- (1,4) -- (1,0) -- (2,0) -- (2,3) -- (3,3) -- (3,1) -- (2,1)
							(1,3) -- (2,3);
				\draw (0.5,3) node {$\hat w$}
							(2.5,2) node {$\eps$};
			\end{tikzpicture}
		\end{displaymath}
		
		This leaves the associativity axiom, which is the identity on the right of \eqref{lax algebra axioms}. Analogous to the argument given before we may equivalently prove it holds after composing it on the left with the composite of the cells $\hat w$, $T\hat w$ and $T^2 \hat w$, as shown on the left below. This follows from the fact that this composite factors as the opcartesian cell defining the companion of $T^2\iota_{\ps A} \of T\iota_{\ps A} \of \iota_{\ps A}$, followed by the composite of the cells $\zeta$, $T\zeta$ and $T^2\zeta$, which can be shown to define $w \of Tw \of T^2 w$ as a left Kan extension, just like we have shown that the composite of $\zeta$ and $T\zeta$ defines $w \of Tw$ as a left Kan extension.
		
		First consider the left-hand side of the associativity axiom, composed on the left with $T^2\hat w$, $T\hat w$ and $\hat w$. It reduces as follows, where the first identity follows from the $T$-image of the identity on the left of \eqref{lax algebra axioms} and the second one from the identity on the left of \eqref{associator definition}; the empty cell in the composite on the right depicts the identity cell $\id_{T^2\ps A} = \mu_{\ps A} \of \mu_{T\ps A} \of T^2\iota_{\ps A} \of T\iota_{\ps A}$.
		\begin{displaymath}
			\begin{tikzpicture}[baseline=+1.7cm, x=0.6cm, y=0.6cm, font=\scriptsize]
				\draw	(4,1) -- (3,1) -- (3,3) -- (4,3) -- (4,0) -- (5,0) -- (5,3) -- (6,3) -- (6,1) -- (5,1)
							(3,3) -- (3,4) -- (2,4) -- (2,1) -- (1,1) -- (1,6) -- (0,6) -- (0,0) -- (1,0) -- (1,1)
							(1,5) -- (2,5) -- (2,4)
							(2,2) -- (3,2)
							(4,2) -- (5,2);
				\draw (0.5,3) node {$\hat w$}
							(1.5,3) node {$T\hat w$}
							(2.5,3) node {$T^2\hat w$}
							(3.5,2) node {$T\bar w$}
							(4.5,1) node {$\bar w$};
			\end{tikzpicture} \mspace{9mu} = \mspace{9mu} \begin{tikzpicture}[baseline=+1.7cm, x=0.6cm, y=0.6cm, font=\scriptsize]
				\draw	(3,1) -- (4,1) -- (4,3) -- (3,3) -- (3,0) -- (2,0) -- (2,5) -- (1,5) -- (1,0) -- (0,0) -- (0,6) -- (1,6) -- (1,5)
							(1,1) -- (2,1)
							(2,4) -- (3,4) -- (3,3)
							(2,2) -- (3,2);
				\draw	(0.5,3) node {$\hat w$}
							(1.5,3) node {$T\hat w$}
							(2.5,1) node {$\bar w$};
			\end{tikzpicture} \mspace{9mu} = \mspace{9mu} \begin{tikzpicture}[baseline=+1.7cm, x=0.6cm, y=0.6cm, font=\scriptsize]
				\draw (1,5) -- (1,0) -- (0,0) -- (0,6) -- (1,6) -- (1,5) -- (2,5) -- (2,1) -- (1,1);
				\draw (0.5,3) node {$\hat w$};
			\end{tikzpicture}
		\end{displaymath}
		Analogously the right-hand side of the associativity axiom, also composed on the left with $T^2\hat w$, $T\hat w$ and $\hat w$, reduces as shown below. Here the first identity follows from the naturality of $\iota$ with respect to $T^2\hat w$, while the last two follow from the identity on the left of \eqref{associator definition}. We conclude that the two sides of the identity forming the associativity axiom coincide after being composed with $T^2\hat w$, $T\hat w$ and $\hat w$; as previously discussed this implies the axiom itself.
		\begin{displaymath}
			\begin{tikzpicture}[baseline=+1.7cm, x=0.6cm, y=0.6cm, font=\scriptsize]
				\draw	(4,2) -- (5,2) -- (5,0) -- (4,0) -- (4,3) -- (3,3) -- (3,0) -- (2,0) -- (2,5) -- (1,5) -- (1,0) -- (0,0) -- (0,6) -- (1,6) -- (1,5)
							(1,1) -- (2,1)
							(2,4) -- (3,4) -- (3,3)
							(2,2) -- (3,2)
							(3,1) -- (4,1);
				\draw	(0.5,3) node {$\hat w$}
							(1.5,3) node {$T\hat w$}
							(2.5,1) node {$\bar w$}
							(2.5,3) node {$T^2\hat w$}
							(4.5,1) node {$\bar w$};
			\end{tikzpicture} \mspace{9mu} = \mspace{9mu} \begin{tikzpicture}[baseline=+1.7cm, x=0.6cm, y=0.6cm, font=\scriptsize]
				\draw (4,3) -- (4,4) -- (5,4) -- (5,0) -- (4,0) -- (4,3) -- (3,3) -- (3,0) -- (2,0) -- (2,5) -- (1,5) -- (1,0) -- (0,0) -- (0,6) -- (1,6) -- (1,5)
							(1,1) -- (2,1)
							(2,4) -- (3,4) -- (3,3)
							(3,1) -- (4,1)
							(4,2) -- (5,2);
				\draw	(0.5,3) node {$\hat w$}
							(1.5,3) node {$T\hat w$}
							(2.5,2) node {$\bar w$}
							(3.5,2) node {$T\hat w$}
							(4.5,1) node {$\bar w$};
			\end{tikzpicture} \mspace{9mu} = \mspace{9mu} \begin{tikzpicture}[baseline=+1.7cm, x=0.6cm, y=0.6cm, font=\scriptsize]
				\draw (1,5) -- (2,5) -- (2,1) -- (1,1) -- (1,3) -- (2,3)
							(1,3) -- (1,6) -- (0,6) -- (0,0) -- (1,0) -- (1,1)
							(2,5) -- (2,6) -- (3,6) -- (3,0) -- (2,0) -- (2,1)
							(2,2) -- (3,2)
							(2,4) -- (3,4);
				\draw (0.5,3) node {$\hat w$}
							(1.5,2) node {$T\hat w$}
							(1.5,4) node {$\eps$}
							(2.5,1) node {$\bar w$};
			\end{tikzpicture} \mspace{9mu} = \mspace{9mu} \begin{tikzpicture}[baseline=+1.7cm, x=0.6cm, y=0.6cm, font=\scriptsize]
				\draw (1,5) -- (1,0) -- (0,0) -- (0,6) -- (1,6) -- (1,5) -- (2,5) -- (2,1) -- (1,1);
				\draw (0.5,3) node {$\hat w$};
			\end{tikzpicture}
		\end{displaymath}
		
		This concludes the proof of $(w, \bar w, \hat w)$ forming a lax $T$-algebra structure on the presheaf object $\ps A$. The last assertion, on the invertibility of $\hat w$, is proven as follows. If $\map{\iota_{\ps A}}{\ps A}{T\ps A}$ is full and faithful then $\hat w$ is invertible by applying \lemref{pointwise left Kan extension along full and faithful map} to the identity that defines $\hat w$, on the left of \eqref{definition unitor}. On the other hand if $\hat w$ is invertible then $\ps A$ forms a pseudo $T$-algebra so that, $T$ being colax-idempotent, $\iota_{\ps A}$ is left-adjoint to $w$ in an adjunction whose unit is $\hat w$; that $\iota_{\ps A}$ is full and faithful then follows from \lemref{adjunction with invertible unit}.
	\end{proof}
	\bibliographystyle{alpha}
	\bibliography{main}
\end{document}